\documentclass[11pt,reqno]{siamltex}
\usepackage{amsmath,amsfonts,amssymb,mathrsfs}
\usepackage{amssymb,mathrsfs,graphicx}

\usepackage{ifthen} 
\provideboolean{shownotes} 
\setboolean{shownotes}{true} 
\newcommand{\margnote}[1]{
\ifthenelse{\boolean{shownotes}}%
{\marginpar{\raggedright\tiny\texttt{#1}}}%
{}%
}
\newcommand{\hole}[1]{
\ifthenelse{\boolean{shownotes}}%
{\begin{center} \fbox{ \rule {.25cm}{0cm}
\rule[-.1cm]{0cm}{.4cm} \parbox{.85\textwidth}{\begin{center}
\texttt{#1}\end{center}} \rule {.25cm}{0cm}}\end{center}}
{}
}

\title{Relative entropy in diffusive relaxation}

\author{Corrado Lattanzio\thanks{
Dipartimento di Matematica Pura ed Applicata
\newline
Universit\`a degli Studi dell'Aquila
\newline
Via Vetoio
\newline
I 67010 Coppito (L'Aquila) AQ 
\newline 
Italy
({\tt corrado@univaq.it}).}
\and
Athanasios E. Tzavaras\thanks{
Department of Applied Mathematics
\newline
University of Crete 
\newline
GR 71409 Heraklion, Crete 
\newline
Greece 
\newline
and
\newline
Institute for Applied and Computational Mathematics
\newline
Foundation for Research and Technology
\newline
GR 70013 Heraklion, Crete
\newline
Greece 
({\tt tzavaras@tem.uoc.gr}).}
}

\numberwithin{equation}{section}

\newtheorem{remark}[theorem]{Remark}

\newcommand{\R}{\mathbb R}
\newcommand{\T}{\mathbb T}

\newcommand{\e}{\varepsilon}
\newcommand{\eps}{\varepsilon}

\newcommand{\dx}{\mathop{{\rm div}_{x}}}

\newcommand{\px}{\partial_{x}}
\newcommand{\pt}{\partial_{t}}
\newcommand{\pxi}{\partial_{x_{i}}}

\def\charf {\mbox{{\text 1}\kern-.30em {\text l}}}

\def\del{\partial}

\begin{document}
\allowdisplaybreaks

\maketitle

\begin{abstract}
We establish convergence in the diffusive limit from entropy weak solutions of 
the equations of compressible gas dynamics with friction to the porous media equation away from vacuum.
The result is based on a Lyapunov type of functional provided by a calculation of the relative entropy.
The relative entropy method is also employed to establish convergence from entropic weak solutions
of viscoelasticity with memory  to the system of viscoelasticity of the rate-type. 
\end{abstract}

 
 \section{Introduction}

The relative entropy method of Dafermos and DiPerna \cite{Dafermos79a,Dafermos79,Diperna79} 
provides an efficient mathematical tool for studying stability and limiting processes
among thermomechanical theories. It is intimately connected to the second law of thermodynamics
and has been tested in various situations involving stability and asymptotic behavior of shocks
({\it e.g.} \cite{Diperna79,CFL02,LV11}),
relaxation or kinetic limits in the hydrodynamic regime \cite{Tza05,BV05},
stability and limiting processes among thermomechanical theories
\cite{Dafermos79a,Iesan94,LT06,DST12}.

The method hinges on a direct calculation of the relative entropy between a  dissipative 
solution and an entropy conservative (smooth) solution for the underlying thermomechanical process,
which provides a remarkable stability formula \cite{Dafermos79a,Dafermos79}. In more complicated situations involving
the comparison of two solutions with shocks it is supplemented with additional information,
{\it e.g.} \cite{Diperna79,CFL02,LV11}.
The objective of this article is to extend the relative entropy formula in situations 
where a dissipative solution of a thermomechanical system is directly compared to a dissipative solution of 
a limiting system. We use as test cases various paradigms of diffusive limits, the most significant perhaps
being the validation of the limit from the Euler equations with friction
to the porous media equation in the zero-relaxation limit.

We consider the system of isentropic gas dynamics with friction
\begin{equation}
\begin{aligned}
        \rho_{t} +\frac{1}{\e}\dx m &= 0
        \\
	m_{t} + \frac{1}{\e}\dx \frac{m\otimes m}{\rho} 
       + \frac{1}{\e}\nabla_{x}p(\rho) 
	&= -\frac{1}{\e^{2}}m
\end{aligned}
\label{intro:euler}
\end{equation}
with the so called diffusive scaling, which captures the effective long-time response.
In the limit $\eps \to 0$ this system approaches the porous media equation
\begin{equation}
    \rho_{t} - \triangle_{x}p(\rho) = 0.
    \label{intro:pme}
\end{equation}
This problem has served as a paradigm for the theory of diffusive relaxation \cite{MR00,LY01,DM04}
and has been justified either by asymptotic in time analysis \cite{HL92,Nis96,Liu97,HMP05,HPW11},
or via direct analysis of the relaxation limit, for weak solutions in  \cite{MMS88,MM90,MR00}
or for smooth solutions near equilibrium in  \cite{CG07,LC11}.

In this paper we compare directly a weak entropy solution of \eqref{intro:euler} 
to a smooth solution of \eqref{intro:pme} using a relative entropy analysis (Proposition  \ref{prop:relenEuler}).
This, in turn, provides a convergence result to solutions of the porous media equation that stay 
away from vacuum (Theorems \ref{th:finalEuler3d} and \ref{th:finalEuler1d}). 
The novelty of the present work is the simplicity of the proof following a Lyapunov type of analysis;
in addition some new situations are analyzed (for instance, solutions approaching different end-states at $\pm \infty$),
plus a rate of convergence is obtained.
Finally, in the spirit of \cite{BDS11,DST12}, the relative entropy inequality is extended 
between entropy measure-valued solutions
of the Euler equation  and the porous media in Section \ref{subseq:relenmv}.

We then test some other cases of diffusive relaxation using the relative entropy method.
In Section \ref{sec:psystem}, we consider
the $p$-system with damping in Lagrangian coordinates and establish convergence to 
a parabolic equation, in the high-friction limit (Theorem \ref{th:finalpsys}). 
In Section \ref{sec:viscoelasticity}, we consider the limiting process  from viscoelasticity 
of the memory type \eqref{eq:mainVE} to the system of viscoelasticity of the rate-type \eqref{eq:equilVE}
in the diffusive regime. We provide a relative entropy estimation between the two theories and a convergence result 
(see Proposition \ref{prop:relenVE} and Theorem \ref{th:finalVE}) thereby extending for quasilinear systems previous convergence results in the semilinear case from \cite{DiFL04,DL09}.

It is remarkable that in all those examples the dissipation of the approximating
system can be split in two separate parts: the dissipation of the limit diffusion equation,
and a second part that captures the dissipation of the approximating system 
relative to its diffusive-scale limit.


\section{Isentropic gas dynamics in Eulerian coordinates with damping}
\label{sec:euler}
We consider the system of isentropic gas dynamics in three 
space dimensions with a damping term:
\begin{equation}
    \begin{cases}
        \displaystyle{\rho_{t} +\frac{1}{\e}\dx m =0}&   \\
	& \\
       \displaystyle{ m_{t} + \frac{1}{\e}\dx \frac{m\otimes m}{\rho} 
       + \frac{1}{\e}\nabla_{x}p(\rho) 
	= -\frac{1}{\e^{2}}m}, & 
    \end{cases}
    \label{eq:euler}
\end{equation}
where $t\in\R$, $x\in\R^{3}$, the density $\rho\geq 0$ and the momentum flux $m\in\R^{3}$. 
The pressure $p(\rho)$ satisfies $p'(\rho) > 0$ which makes the system hyperbolic.
An important particular case is that of the $\gamma$--law: $p(\rho) = k\rho^\gamma$ with $\gamma \geq 1$ and $k>0$.
In \eqref{eq:euler},  the variables $(x,t)$ are already scaled in the so called  diffusive scaling. 
In the diffusive relaxation limit $\e \to 0$, solutions of \eqref{eq:euler} formally converge to the 
porous media equation
\begin{equation}
    \bar \rho_{t} - \triangle_{x}p(\bar \rho) = 0.
    \label{eq:pme}
\end{equation}
The goal of this work is to study this limit via the relative entropy method.

We recall that $(\eta , q_1 , q_2 , q_3 ) (\rho, m) : \R^+ \times \R^3 \to \R \times \R^3$ is an entropy--entropy flux pair
for the hyperbolic system \eqref{eq:euler} if it satisfies the differential relations:
\begin{equation}
\begin{aligned}
\frac{\del q_j}{\del m_i} &= 
\Big ( \frac{\del \eta}{\del \rho} + \frac{1}{\rho} m_k \frac{\del \eta}{\del m_k} \Big ) \delta_{ij} + \frac{1}{\rho} m_j 
\frac{\del \eta}{\del m_i}
\\
\frac{\del q_j}{\del \rho} &= - \frac{1}{\rho^2} m_i m_j \frac{\del \eta}{\del m_i} + p'(\rho) \frac{\del \eta}{\del m_j }
\end{aligned}
\end{equation}
where $i,j = 1, 2, 3$,  $\delta_{ij}$ stands for the Kronecker symbol, and the summation convention is used.
Moreover, the entropy $\eta(\rho, m)$ is dissipative (for the underlying relaxation process) if
\begin{equation*}
\nabla_{(\rho,m)}\eta(\rho,m) \cdot (0, - m) = - \nabla_{m}\eta(\rho,m) \cdot m \leq 0.
\end{equation*}

An example of an entropy pair is provided by the mechanical energy 
\begin{equation}
\label{eq:mechen}
   \eta(\rho,m) =  \frac{1}{2}\frac{|m|^{2}}{\rho} + h(\rho),
\end{equation}
and the associated flux of mechanical work
\begin{equation}
\label{eq:mechflux}
    q(\rho,m) = \frac{1}{2}m\frac{|m|^{2}}{\rho^{2}} + 
    mh'(\rho) \, .
\end{equation}
Here, 
$h(\rho) = \rho e(\rho)$,  where $e(\rho)$ is the internal energy of the gas connected to the pressure via
$e' (\rho) = \frac{p(\rho)}{\rho^2}$. Accordingly, 
\begin{equation}
\label{eq:deffunh}
    h''(\rho) = \frac{p'(\rho)}{\rho}; \quad  \rho h'(\rho) = p(\rho) + h(\rho).
\end{equation}
For the particular case of $\gamma$--law gases, $h$ takes the form
\begin{equation*}
    h(\rho) = 
    \begin{cases}
    \displaystyle{\frac{k}{\gamma-1}\rho^{\gamma} = \frac{1}{\gamma-1}p(\rho)} & \hbox{for}\ \gamma>1; \\
      \displaystyle{k\rho\log\rho  } & \hbox{for}\ \gamma=1.
    \end{cases}
\end{equation*}
Smooth solutions of \eqref{eq:euler} satisfy the identity
\begin{equation}
\label{eq:strentrEuler}
    \eta(\rho,m)_{t} +\frac{1}{\e}\dx q(\rho, m) =  
    -\frac{1}{\e^{2}}
    \nabla_{m}\eta(\rho,m)\cdot m = 
    - \frac{1}{\e^{2}}\frac{|m|^{2}}{\rho}\leq0,
\end{equation}
The mechanical energy $\eta(\rho, m)$ is dissipative for the relaxation process \eqref{eq:euler}.

\subsection{Hilbert expansion}\label{subsec:hilbertEuler}
We start by reviewing the Hilbert expansion associated to the relaxation process from \eqref{eq:euler} to
\eqref{eq:pme}. We introduce the asymptotic expansions
 \begin{align*}
      \rho &= \rho_{0} + \e \rho_{1} + \e^{2}\rho_{2}+ \ldots \, , \\
      m  &= m_{0} + \e m_{1} + \e^{2}m_{2}+ \ldots  \, ,
\end{align*}
to the balance of mass and momentum equations in \eqref{eq:euler},
and collect together the terms of the same order, to obtain, respectively, 
\begin{align*}
&\mbox{$O(\e^{-1})$} \quad &\dx m_{0} = 0 \, ,
\\
&\mbox{$O(1)$}    &\pt \rho_{0} + \dx m_{1} = 0 \, ,
\\
&\mbox{$O(\e)$}   &\pt \rho_{1} + \dx m_{2} = 0 \, ,
\end{align*}
and
\begin{align*}
&\mbox{$O(\e^{-2})$} \quad &&\quad m_{0} = 0 \, , 
\\
&\mbox{$O(\e^{-1})$}    &&-m_1 = \nabla_{x} p(\rho_{0}) \, ,
\\
&\mbox{$O(1)$}   &&-m_2 =   \nabla_{x} (p'(\rho_{0})\rho_{1}) \, ,
\\
&\mbox{$O(\e)$}   &&-m_3 = \partial_t m_1 + \dx \Big (  \frac{m_1\otimes m_1}{\rho_0}\Big ) + \nabla_x 
\Big ( p'(\rho_0)\rho_2 +\frac{1}{2} p''(\rho_0)\rho_1^2 \Big) \, .
\end{align*}
In particular, we recover the equilibrium relation $m_{0} = 0$ for the state variables,
Darcy's law $m_{1} = -\nabla_{x}p(\rho_{0})$, and observe that $\rho_0$ satisfies \eqref{eq:pme}.

Next, we focus on the asymptotic expansion of the entropy equation \eqref{eq:strentrEuler},
and in particular on how the hyperbolic entropy (the mechanical energy) captures in the $\eps \to 0$ limit 
the entropy structure of the porous media equation.
Introducing the Hilbert expansion into \eqref{eq:strentrEuler} and using  $m_0=0$, 
we see that
\begin{align*}
     &  \partial_t \big( h(\rho_{0}) + \e h'(\rho_0)\rho_1  \big) +  
     \frac{1}{\e} \dx \big (\e m_{1}h'(\rho_{0})  + \e^2(m_2h'(\rho_0) + m_1h''(\rho_0)\rho_1) \big ) 
      \\
     &= - \frac{|m_1|^2}{\rho_0} + \e \left ( |m_1|^2 \frac{\rho_1}{\rho_0^2}  - 2\frac{m_1\cdot m_2}{\rho_0}\right) + O(\e^2).
\end{align*}
Again, collecting together terms of the same order gives
\begin{align*}
\mbox{$O(1)$} \quad  &h(\rho_{0})_t + \dx \big ( m_{1}h'(\rho_{0})  \big) = - \frac{|m_1|^2}{\rho_0} \, ,
\\
\mbox{$O(\e)$} \; \; \;
&\partial_t  \big(h'(\rho_0)\rho_1  \big)  + \dx \big (m_2h'(\rho_0) + m_1h''(\rho_0)\rho_1\big ) =
 |m_1|^2 \frac{\rho_1}{\rho_0^2}  - 2\frac{m_1\cdot m_2}{\rho_0} \, .
\end{align*}
Since $m_{1} = -\nabla_{x}p(\rho_{0})$, the leading order term $\rho_0$ in the
diffusive limit satisfies the energy identity
\begin{equation}\label{eq:entrPME}
    h(\rho)_{t} - \dx \big ( h'(\rho) \nabla_{x}p(\rho) \big) = - 
    \frac{| \nabla_{x} p(\rho)|^{2}}{\rho} \, .
\end{equation}

Equation \eqref{eq:entrPME} captures the entropy dissipation of the porous medium equation \eqref{eq:pme} 
and $h(\rho)$ is indeed the entropy selected by Otto \cite{Ott01} in his 
gradient flow interpretation of \eqref{eq:pme}.

\subsection{Relative entropy identity}\label{subsec:relenexEuler}
Let $(\rho^\eps, m^\eps)$  be a weak solution of \eqref{eq:euler} that satisfies the weak form of
the entropy inequality
\begin{equation}
    \eta(\rho,m)_{t} +\frac{1}{\e}\dx q(\rho, m)  + \frac{1}{\e^{2}}\frac{|m|^{2}}{\rho}\leq 0.
    \label{eq:entrEuler}
\end{equation}
(We drop the $\eps$-dependence of $(\rho^\eps, m^\eps)$ except where emphasis makes it necessary.)
Let $\bar \rho$ be a smooth solution of the porous media equation \eqref{eq:pme}; such a solution will
also satisfy the entropy identity \eqref{eq:entrPME}.
Our goal is to devise an identity  that  monitors the distance between
$\rho^\eps$ and $\bar \rho$.

Such identities have been obtained via the relative entropy method in the context of problems 
of hyperbolic relaxation \cite{LT06,Tza05,BV05}. The relative entropy is defined as the
quadratic part of the Taylor series expansion between two solutions $(\rho, m)$ and $(\bar \rho, \bar m)$;
it takes the form
\begin{align}
    \eta(\rho, m \left | \bar \rho, \bar m \right. ) 
    &:= \eta (\rho, m) - \eta (\bar \rho, \bar m) - 
    \eta_{\rho} (\bar\rho, \bar m) (\rho - \bar\rho)
    \nonumber
\\
    &\qquad \quad 
    - \nabla_{m}\eta (\bar\rho, \bar m) \cdot (m - \bar m) 
    \nonumber
\\
    &= \frac{1}{2}\rho \left |\frac{m}{\rho} - \frac{\bar m}{\bar \rho}  \right |^2 + h(\rho \left | \bar \rho \right.)\, ,
\label{eq:relendef}
\end{align}
while the corresponding relative entropy-flux reads
\begin{align}
    &q_{i}(\rho, m \left | \bar \rho, \bar m \right. )  := q_{i} (\rho, m) - q_{i}(\bar\rho, \bar m) - 
    \eta_{\rho} (\bar\rho, \bar m) (m_{i}- \bar m_{i}) 
    \nonumber\\
   &\qquad \qquad \qquad \qquad
    -\nabla_{m}\eta(\bar \rho, \bar m) \cdot (f_{i}(\rho,m) - f_{i}(\bar 
    \rho, \bar m)) 
    \nonumber\\
    &= \frac{1}{2}m_i \Big |\frac{m}{\rho} - \frac{\bar m}{\bar \rho}  \Big |^2 + \rho (h'(\rho) - 
    h'(\bar \rho))\Big ( \frac{m_i}{\rho} - \frac{\bar m_i}{\bar \rho} \Big ) + \frac{\bar m_i}{\bar \rho} h(\rho \left | \bar \rho \right.)\, ,
    \label{eq:fluxrelEuler}
\end{align}
where $i=1,2,3$, $f_i$ stands for the (vector) of the flux in \eqref{eq:euler},
\begin{equation}
    f_{i}( \rho,  m) =  m_{i}\frac{m}{\rho} + p(\rho) I_{i}
\label{eq:fluxEuler}
\end{equation}
and $I_{i}$ is the $i$--th column of the $3\times 3$ identity 
matrix. 

Now the question arises about how to select $\bar m$ in \eqref{eq:relendef}, \eqref{eq:fluxrelEuler}. 
This relates to a significant
difference among the hyperbolic relaxation and the diffusive relaxation frameworks:
in the existing studies of hyperbolic relaxation limits
one compares an \emph{energy dissipative} with an \emph{energy conservative} solution. The fact that the
limiting solution is energy conservative (and smooth) is an important restriction
 in the derivation of the relative entropy
identities available in the hyperbolic relaxation framework (see \cite{Tza05,BV05}).  By contrast, 
by the   nature of the diffusive relaxation framework, the solutions to be compared
have both to be energy dissipative.
To effect the comparison we select an $\eps$-dependent solution
$(\bar \rho, \bar m)$ that adapts itself in the relaxation process.

A suitable selection of $\bar m$ is proposed by rewriting \eqref{eq:pme} in the form,
\begin{equation}
\label{eq:pmeref}
\begin{aligned}
\bar \rho_{t} +\frac{1}{\e}\dx \bar m &= 0
\\
\bar m = - \eps \nabla_{x}p(\bar \rho) \, ,
\end{aligned}
\end{equation}
of the conservation of mass equation in \eqref{eq:euler} together with (a rescaled form of) Darcy's law.
The energy identity \eqref{eq:entrPME} may also be expressed in terms of $(\bar \rho, \bar m)$ as
\begin{equation*}
    h(\bar \rho)_{t} +\frac{1}{\eps} \dx \big ( \bar m h'(\bar \rho)  \big) = - 
    \frac{1}{\eps^2} \frac{| \bar m|^{2}}{\bar \rho} \, .
\end{equation*}
In turn, \eqref{eq:pmeref} is embedded into
the system of \emph{Euler equations  with relaxation}, plus additional terms purported
to be higher-order errors. A simple calculation shows that $(\bar \rho, \bar m)$ satisfies
\begin{equation}
    \begin{aligned}
	\bar \rho_{t} +\frac{1}{\e}\pxi \bar m_{i} &=0 
	\\
       \bar m_{t} + \frac{1}{\e} \pxi f_{i}(\bar \rho, \bar m)
	&= -\frac{1}{\e^{2}}\bar m + e(\bar \rho, \bar m) \, ,
    \end{aligned}
    \label{eq:eulerBAR}
\end{equation}
where (we use the convention of summation over repeated indices and) $\bar e$  is  given by 
\begin{align}
   \bar e :=  e(\bar \rho, \bar m) &= \frac{1}{\e} \dx \left (\frac{\bar 
       m\otimes \bar m}{ \bar\rho} \right ) - \e \pt \nabla_{x}p(\bar 
       \rho) \nonumber\\
       &= 
       \e  \dx \left (\frac{ 
       \nabla_{x}p(\bar \rho)\otimes \nabla_{x}p(\bar \rho)}{\bar \rho} \right ) 
       -\e \nabla_{x}(p'(\bar 
       \rho)\triangle_{x}p(\bar \rho)) \nonumber\\
       &= O(\e)
       \label{eq:errortermdef}
\end{align}
and is thus an error term.

\bigskip
The main result of this section is:

\begin{proposition}\label{prop:relenEuler}
Let $(\rho,m)$ be a weak entropy solution of \eqref{eq:euler} satisfying \eqref{eq:entrEuler}
and let  $(\bar \rho, \bar m)$ be a smooth solution of \eqref{eq:pmeref}. Then, 
\begin{equation}
 \eta(\rho, m \left | \bar \rho, \bar m \right. )_t + \frac{1}{\e}\dx q(\rho, m\left | \bar \rho, \bar m \right. ) \leq -\frac{1}{\e^2}
 R(\rho, m\left | \bar \rho, \bar m \right. ) - Q -E,
    \label{eq:relenEuler}
    \end{equation}
where
\begin{align}
&R(\rho, m\left | \bar \rho, \bar m \right. ) = \rho \left |\frac{m}{\rho} - \frac{\bar m}{\bar \rho}  \right |^2 \, ,
\nonumber \\
\label{eq:defrhs}
& Q =  
    \frac{1}{\e} \nabla^{2}_{(\rho,m)} \eta(\bar \rho, \bar m) 
    \begin{pmatrix} \bar \rho_{x_i}  \\  \bar m_{x_i} \end{pmatrix}  
    \cdot 
    \begin{pmatrix} 0  \\  f_i (\rho , m | \bar \rho , \bar m ) \end{pmatrix}  \, ,
    \\
& E =   e(\bar \rho, \bar m)\cdot \frac{\rho}{\bar\rho}\left ( \frac{m}{\rho} - \frac{\bar m}{\bar \rho} \right ) \, ,
\nonumber
\end{align}    
and $e(\bar \rho, \bar m)$ is defined in \eqref{eq:errortermdef}.
\end{proposition}

\begin{remark}\label{rem:relenEulerQE}
The following remarks are in order, concerning the terms appearing on the right of \eqref{eq:relenEuler}:

\medskip
\noindent
(a) The coefficient of the quadratic term $Q$ depends only on $(\bar \rho, \bar m)$; it is explicitly given by
 \begin{equation*}
     \frac{1}{\e}\Big(\eta_{\rho m_{j}}(\bar \rho, \bar m)\bar 
	 \rho_{x_{i}}
	 + \eta_{m_{k}m_{j}}(\bar \rho, \bar m)\pxi \bar m_{k} \Big) = 
	  \frac{1}{\e}\partial_{x_i}\left (\frac{\bar m_j}{\bar\rho} \right) = 
	  -\partial_{x_ix_j}h'(\bar\rho) \, ,
 \end{equation*}
and is thus of $O(1)$ in $\eps$.

\medskip
\noindent
(b) Since $e(\bar \rho, \bar m) = O(\e)$, the coefficient of the  error term $E$ is of $O(\e)$.

\medskip
\noindent
(c) The term  $R(\rho, m\left | \bar \rho, \bar m \right. )$ captures the dissipation of the
relaxation system \eqref{eq:euler} relative to its diffusive scale limit \eqref{eq:pme}.
It turns out to be the  \emph{quadratic part of the dissipative relaxation term with respect 
to $(\bar \rho, \bar m)$}.
Indeed, for
$R(\rho, m) = \nabla_{m}\eta\cdot m = \frac{|m|^{2}}{\rho}$, $R_{\rho}  = \eta_{\rho m_{j}}  m_{j}$,
 $R_{m_{i}}  =  \eta_{m_{i}} + \eta_{m_{i}m_{j}} \bar m_{j}$. Using 
 \begin{equation*}
    \nabla_m \eta(\rho,m) = \frac{m}{\rho},\  \eta_{\rho m_{j}}( \rho,  m)= -\frac{m_{j}}{\rho^{2}}, \  
     \eta_{m_{k}m_{j}}( \rho,  m) = \frac{1}{\rho}\delta_{kj},
 \end{equation*}
we compute the Hessian of $R$,
\begin{equation*}
    \nabla^{2}_{(\rho, m)}R(\rho, m) = 
    \begin{pmatrix}
       \displaystyle{ 2\frac{|m|^{2}}{\rho^{3}}} &
       \displaystyle{ - 2\frac{m_{j}}{\rho^{2}}}
       \\
	\displaystyle{- 2\frac{m_{i}}{\rho^{2}}} & 
	\displaystyle{\frac{2}{\rho}\delta_{ij}}
    \end{pmatrix}_{i,j=1,2,3}
\end{equation*}
and see that it has eigenvalues
\begin{equation*}
    \lambda_{1}=0,\ \lambda_{2,3} = \frac{2}{\rho}> 0,\ \lambda_{4} = 
    \frac{2}{\rho} + \frac{2}{\rho^{3}}|m|^{2}>0 \, ,
\end{equation*}
and is thus positive semidefinite for $\rho$, $\bar \rho >  0$. A direct computation also shows
that $R(\rho, m\left | \bar \rho, \bar m \right. ) = \rho \left |\frac{m}{\rho} - \frac{\bar m}{\bar \rho}  \right |^2$
and justifies the notation used in \eqref{eq:defrhs}$_1$.
The property $R(\rho, m\left | \bar \rho, \bar m \right. ) \geq 0$ is instrumental in the
forthcoming stability analysis; a similar property holds for all examples treated in this article.
\end{remark}

\begin{proof}[Proof of Proposition \ref{prop:relenEuler}]
By hypothesis $(\rho, m)$ satisfies the weak from of the entropy inequality
\begin{equation}
\label{eq:entro}
    \eta(\rho,m)_{t} +\frac{1}{\e}\dx q(\rho, m) \leq - \frac{1}{\e^{2}}
    \nabla_{m}\eta(\rho,m)\cdot m = 
    - \frac{1}{\e^{2}}\frac{|m|^{2}}{\rho} \, .
\end{equation}
Also, $(\bar \rho, \bar m)$ satisfies the energy identity
\begin{align}
    \eta(\bar \rho,\bar m)_{t} +\frac{1}{\e}\dx q(\bar\rho,  \bar m)& = 
    -\frac{1}{\e^{2}}
    \nabla_{m}\eta(\bar \rho, \bar m)\cdot \bar m + 
    \nabla_{m}\eta(\bar \rho, \bar m) \cdot e(\bar \rho, \bar m) \nonumber\\
    &= 
    - \frac{1}{\e^{2}}\frac{|\bar m|^{2}}{\bar \rho}  + 
    \nabla_{m}\eta(\bar \rho, \bar m) \cdot \bar e .
    \label{eq:smoothentr}
\end{align}

From \eqref{eq:euler} and \eqref{eq:eulerBAR} we obtain
\begin{equation}
\label{eq:diff}
\begin{aligned}
( \rho - \bar \rho)_{t} +\frac{1}{\e}  \pxi (m_i - \bar m_i ) &= 0
\\
( m - \bar m)_{t} + \frac{1}{\e} \pxi( f_i (\rho, m) -  f_{i}(\bar \rho, \bar m) )
	&= -\frac{1}{\e^{2}} ( m - \bar m )  - \bar e
\end{aligned}
\end{equation}
and use the smoothness of $(\bar \rho, \bar m)$ and \eqref{eq:eulerBAR} to compute
\begin{align}
    & \pt \Big (\eta_{\rho}(\bar \rho,\bar m)(\rho - \bar \rho)  + 
    \nabla_{m}\eta (\bar \rho,\bar m) \cdot (m - \bar m)  
    \Big )  \nonumber
\\
    &\quad + \frac{1}{\e} \pxi \Big (\eta_{\rho}(\bar \rho,\bar m) 
    (m_{i} - \bar m_{i}) +  \nabla_{m}\eta (\bar \rho,\bar m) \cdot  
    (f_{i}(\rho,m) - f_{i}(\bar 
	\rho, \bar m) \Big) \nonumber
\\
&\ =
    \nabla_{m}\eta(\bar \rho, \bar m)\cdot  \left ( -\frac{1}{\e^{2}} (m -\bar m)  - \bar e \right )
    \nonumber
\\
    &\qquad + \pt \big (\eta_{\rho}(\bar \rho,\bar m)\big )(\rho - \bar 
    \rho) + \pt \big (  \nabla_{m}\eta (\bar \rho,\bar m) \big)\cdot 
    (m-\bar m) \nonumber
\\
    &\qquad + \frac{1}{\e} \pxi\big (\eta_{\rho}(\bar \rho,\bar 
    m)\big ) (m_{i} - \bar m_{i}) + \frac{1}{\e} \pxi \big ( \nabla_{m}\eta (\bar 
    \rho,\bar m)  \big)\cdot  (f_{i}(\rho,m) - f_{i}(\bar 
	\rho, \bar m)  \nonumber
\\
&\ = -\frac{1}{\e^{2}}
    \nabla_{m}\eta(\bar \rho, \bar m)\cdot (m -\bar m) - \nabla_{m}\eta(\bar \rho, \bar m) \cdot
    \bar e \nonumber
\\
    &\qquad + \nabla^{2}_{(\rho,m)}\eta(\bar \rho, \bar m)
     \begin{pmatrix} \bar \rho_t  \\  \bar m_t \end{pmatrix}
      \cdot 
      \begin{pmatrix} \rho - \bar \rho  \\ m -\bar m \end{pmatrix} 
      \nonumber
\\
    &\qquad  
   + \frac{1}{\e} \nabla^{2}_{(\rho,m)}\eta(\bar \rho, \bar m) 
    \begin{pmatrix} \bar \rho_{x_i}  \\  \bar m_{x_i} \end{pmatrix}  
    \cdot 
    \begin{pmatrix} m_i - \bar m_i  \\  f_{i}(\rho,m) - f_{i}(\bar \rho, \bar m) \end{pmatrix}
    \nonumber
\\
 &\ = -\frac{1}{\e^{2}}
    \nabla_{m}\eta(\bar \rho, \bar m)\cdot (m -\bar m) - \nabla_{m}\eta(\bar \rho, \bar m) \cdot
    e (\bar \rho, \bar m) \nonumber
\\
    &\qquad + \nabla^{2}_{(\rho,m)}\eta(\bar \rho, \bar m)
     \begin{pmatrix} 0  \\  -\frac{1}{\eps^2} \bar m + e (\bar \rho, \bar m) \end{pmatrix}
      \cdot 
      \begin{pmatrix} \rho - \bar \rho  \\ m -\bar m \end{pmatrix} 
      \nonumber
\\
    &\qquad 
   + \frac{1}{\e} \nabla^{2}_{(\rho,m)} \eta(\bar \rho, \bar m) 
    \begin{pmatrix} \bar \rho_{x_i}  \\  \bar m_{x_i} \end{pmatrix}  
    \cdot 
    \begin{pmatrix} 0  \\  f_i (\rho , m | \bar \rho , \bar m ) \end{pmatrix}
\label{eq:uno}
\\
&\ =: J\, .
\nonumber
\end{align}
To obtain \eqref{eq:uno} we used the identities
\begin{align*}
&\nabla^{2}_{(\rho,m)}\eta(\bar \rho, \bar m) 
        \begin{pmatrix}  \del_{x_i} \bar m_i  \\  \del_{x_i} f_i (\bar \rho, \bar m) \end{pmatrix}
        \cdot \begin{pmatrix} \rho - \bar \rho  \\ m -\bar m \end{pmatrix} 
\\
&=
 \nabla^{2}_{(\rho,m)}\eta(\bar \rho, \bar m) 
\begin{pmatrix}
0 & \delta_{i1} & \delta_{i2} & \delta_{i3} \\
& & \nabla_{(\rho,m)}f_i(\bar \rho, \bar m)   &
\end{pmatrix}
 \begin{pmatrix} \bar \rho_{x_i} \\ \bar m_{x_i}  \end{pmatrix}
 \cdot \begin{pmatrix} \rho - \bar \rho  \\ m -\bar m \end{pmatrix}  
\\
 &\ = 
  \nabla^{2}_{(\rho,m)}\eta(\bar \rho, \bar m)   \begin{pmatrix} \bar \rho_{x_i} \\ \bar m_{x_i}  \end{pmatrix} \cdot 
  \begin{pmatrix}
0 & \delta_{i1} & \delta_{i2} & \delta_{i3} \\
& & \nabla_{(\rho,m)}f_i(\bar \rho, \bar m)   &
\end{pmatrix} 
\begin{pmatrix} \rho - \bar \rho  \\ m -\bar m \end{pmatrix}
\\
&\ =
\nabla^{2}_{(\rho,m)}\eta(\bar \rho, \bar m)   \begin{pmatrix} \bar \rho_{x_i} \\ \bar m_{x_i}  \end{pmatrix} \cdot
\begin{pmatrix}
m_i - \bar m_i \\ \nabla_{(\rho,m)}f_i(\bar \rho, \bar m) \cdot (m - \bar m) 
\end{pmatrix}\, ,
\end{align*}
(resulting from the entropy consistency relations) and the notation
$$
f_i (\rho , m | \bar \rho , \bar m ) =
f_{i}(\rho,m) - f_{i}(\bar \rho, \bar m) 
    -\nabla_{(\rho,m)}f_i(\bar \rho, \bar m) \cdot ( \rho - \bar \rho , m - \bar m) \, .
$$

Finally, combining
 \eqref{eq:relendef}, \eqref{eq:fluxrelEuler}, \eqref{eq:entro}, \eqref{eq:smoothentr} and \eqref{eq:uno},
we conclude
\begin{align*}
    &\eta(\rho, m \left | \bar \rho, \bar m \right. )_{t} +\frac{1}{\e}\dx q(\rho, m \left | \bar \rho, \bar m \right. )\\
    &\leq 
    -\frac{1}{\e^{2}} \Big [  \nabla_{m}\eta(\rho,m)\cdot m -
    \nabla_{m}\eta(\bar \rho, \bar m)\cdot \bar m  - 
    \nabla_{m}\eta(\bar \rho, \bar m)\cdot (m-\bar m)
    \\
     &   \qquad \qquad \quad - \eta_{\rho m_{j}}(\bar \rho,\bar m)\bar m_{j} 
     (\rho - \bar \rho) -  \eta_{m_{k}m_{l}}(\bar \rho, \bar m)\bar m_{l}(m_{k} - \bar m_{k})\Big ]  
     \\
     & \quad \; - \frac{1}{\e} \nabla^{2}_{(\rho,m)} \eta(\bar \rho, \bar m) 
    \begin{pmatrix} \bar \rho_{x_i}  \\  \bar m_{x_i} \end{pmatrix}  
    \cdot 
    \begin{pmatrix} 0  \\  f_i (\rho , m | \bar \rho , \bar m ) \end{pmatrix}
    \\
     & \quad \; - e_{l}(\bar \rho, \bar m) \Big (\eta_{\rho m_{l}}(\bar 
     \rho,\bar m)(\rho - \bar \rho) +  \eta_{m_{k}m_{l}}(\bar \rho, 
     \bar m) (m_{k} - \bar m_{k}) \Big)
     \\
     &=: -\frac{1}{\eps^2} D - Q - E \, ,
 \end{align*}
where
 \begin{equation*}
  E= e_l (\bar\rho,\bar m) \left ( - 
  \frac{\bar m_{l}}{\bar \rho^{2}}(\rho - \bar \rho) + \frac{1}{\bar 
  \rho}(m_{l} - \bar m_{l}) \right )
  = e(\bar\rho,\bar m) \cdot   \frac{\rho}{\bar \rho} \left (\frac{m}{\rho} -  \frac{\bar m}{\bar \rho}\right ) \, ,
  \end{equation*}
\begin{align*}
    &D =  \frac{|m|^{2}}{\rho} -  \frac{|\bar 
    m|^{2}}{\bar \rho} -  \frac{\bar m}{\bar \rho} \cdot (m - \bar m) 
    +  \frac{|\bar m|^{2}}{\bar \rho^{2}} ( \rho - \bar \rho) - 
    \frac{\delta_{k \ell}}{\bar \rho}\bar m_{\ell}\cdot (m_{k} - \bar 
    m_{k}) \\
    &\ = \rho \left |\frac{m}{\rho} - \frac{\bar m}{\bar 
    \rho} \right 
    |^{2}
\end{align*}
and $Q$ is as in \eqref{eq:defrhs}. 
\end{proof}

\medskip
We conclude this section with two lemmas.  The first establishes (under additional hypotheses on $p$) 
a bound of the quadratic term $Q$ in terms of  the relative entropy \eqref{eq:relendef}.

\begin{lemma}\label{lem:controlQ}
Assume that $p(\rho)$ satisfies for some $A > 0$
\begin{equation}
\label{eq:hyppress}
p''(\rho) \le A \frac{p'(\rho)}{\rho} \quad \forall \ \rho > 0 \, .
\tag {A}
\end{equation}
Then  $h(\rho)$ in \eqref{eq:deffunh} verifies, for some  $c>0$,
\begin{equation}
p(\rho \left | \bar \rho \right. ) \leq c h(\rho \left | \bar \rho \right. ) \qquad \forall \ \rho \, , \, \bar \rho > 0 \, .
\label{eq:phcontrol}
\end{equation}
Also there exists a $C>0$ such that for any $i = 1,2,3$,
\begin{equation}
| f_{i}(\rho, m \left | \bar \rho, \bar m \right. ) | \leq  C \eta (\rho, m \left | \bar \rho, \bar m \right. ).
\label{eq:controlQ}
\end{equation}
\end{lemma}

\begin{proof}
Recall that $h'' = \frac{p'}{\rho}$. Hypothesis \eqref{eq:hyppress} implies $p'' \le A h''$.
One easily checks the identity
$$
\begin{aligned}
p( \rho | \bar \rho)  &= p(\rho) - p ( \bar \rho) - p'(\bar \rho) (\rho - \bar \rho)
\\
&= (\rho - \bar \rho)^2 \int_0^1 \int_0^\tau p'' (s\rho + (1-s) \bar \rho) ds d\tau \, ,
\end{aligned}
$$ 
and a similar identity holds for $h(\rho | \bar \rho)$. Then \eqref{eq:phcontrol} follows from $p'' \le A h''$.

A direct computation using \eqref{eq:fluxEuler} shows 
that, for $i , \, j = 1,2,3$,
 \begin{align*}
     f_{i j}(\rho, m \left | \bar \rho, \bar m \right. )  &= 
  \frac{m_{i}m_{j}}{\rho} - \frac{\bar m_{i}\bar m_{j}}{\bar  \rho} 
  + \frac{\bar m_{i}\bar m_{j}}{\bar \rho^{2}}(\rho - \bar \rho) 
  \\
  &\quad - \frac{1}{\bar \rho}(\delta_{i\ell}\bar m_{j} + \bar m_{i} \delta_{j\ell})(m_{\ell} - \bar m_{\ell}))  
  +  p(\rho  \left | \bar \rho\right.)\delta_{ij}
  \\
    &=  \rho \left (\frac{m_i}{\rho} -  \frac{\bar m_i}{\bar \rho}\right )
    \left (\frac{m_j}{\rho} -  \frac{\bar m_j}{\bar \rho}\right ) 
     +  p(\rho  \left | \bar \rho\right.)\delta_{ij}\, ,
      \end{align*}
and \eqref{eq:phcontrol} gives \eqref{eq:controlQ}.
\end{proof}

The second lemma indicates a relation between  the ``metric'' induced by the relative entropy \eqref{eq:relendef}
and more traditional norms.

\begin{lemma}\label{lem:generalconvEuler3d}
Let $h\in C^0[0,+\infty)\cap C^2(0,+\infty)$ satisfies $h''(\rho)>0$ for $\rho>0$ and
\begin{equation}
\label{eq:growthh}
h(\rho) = \frac{k}{\gamma -1} \rho^\gamma + o(\rho^\gamma) \, , \quad \hbox{as}\ \rho\to +\infty
\end{equation}
for some constant $k>0$ and for  $\gamma >1$. 
If $\bar\rho\in K =[\delta, M]$ with $\delta>0$ and $M<+\infty$, then there exist positive constants
$R_0$ (depending on $K$) and $C_1$, $C_2$ (depending on $K$ and $R_0$)
such that
\begin{equation*}
h(\rho \left | \bar \rho \right.) \geq 
\begin{cases}
C_1 |\rho-\bar\rho|^2, &\hbox{for}\ 0\leq \rho\leq R_0,\ \bar\rho\in K, \\
C_2 |\rho-\bar\rho|^\gamma, &\hbox{for}\ \rho> R_0,\ \bar\rho\in K.
\end{cases}
\end{equation*}
\end{lemma}
\begin{proof}
Since $\bar\rho \in K$, there exist positive constants $A$ and $B$ such that for any $0<\rho <+\infty$ 
\begin{equation*}
h(\rho \left | \bar \rho \right.) = h(\rho) - h(\bar \rho) - h'(\bar \rho)(\rho - \bar \rho)
\geq h(\rho) - A - B \rho\, ,
\end{equation*}
\begin{equation*}
\frac{h(\rho \left | \bar \rho \right.)}{|\rho-\bar\rho|^\gamma} \geq \frac{h(\rho) - A - B \rho}{\rho^\gamma}\, .
\end{equation*}
In view of \eqref{eq:growthh}, there exists $R_0$ depending on $K$ such that
\begin{equation*}
h(\rho \left | \bar \rho \right.)  \ge \frac{1}{2} \frac{k}{\gamma -1} |\rho-\bar\rho|^\gamma 
\, , \quad \text{for $\rho > R_0$ and $\bar\rho\in K$}.
\end{equation*}

Consider next the positive function
\begin{align*}
B(\rho, \bar\rho) & := \frac{h(\rho \left | \bar \rho \right.)}{|\rho-\bar\rho|^2} \\
 & =  \int_0^1 \int_0^\tau h'' (s\rho + (1-s) \bar \rho) ds d\tau,  \qquad
  \hbox{$\rho\in [0,R_0]$, $\bar\rho\in K$, $\rho \ne \bar \rho$}
\end{align*}
and note that $\lim_{\rho\to\bar\rho} B(\rho, \bar\rho) = \frac{h''(\bar \rho)}{2}  >0$
and thus $B(\rho, \bar\rho) \in C^0([0,R_0]\times K)$. Hence, we have
\begin{equation*}
h(\rho \left | \bar \rho \right.)\geq m |\rho-\bar\rho|^2 
\quad \text{for $\rho\in [0,R_0]$ and $\bar\rho\in K$}
\end{equation*}
where $m := \min_{\rho\in [0,R_0], \bar\rho\in K}B(\rho, \bar\rho) >0$.
\end{proof}

\begin{remark}\label{rem:phcontrol} Hypothesis \eqref{eq:hyppress} holds for large classes of pressure laws.
Indeed,  \eqref{eq:hyppress} is trivially satisfied for concave increasing pressures. For convex
pressures, one checks that it holds  for $p(\rho) = k \rho^\gamma$ with $\gamma \ge 1$.
By contrast, \eqref{eq:hyppress} is violated for $p(\rho) = e^\rho$.
For a $\gamma$--law gas with $\gamma > 1$, one computes that $h(\rho) = \frac{1}{\gamma -1}p(\rho)$
and thus \eqref{eq:phcontrol} holds as an equality.

Regarding the growth condition \eqref{eq:growthh}, if  $p\in C^0[0,+\infty)\cap C^2(0,+\infty)$ 
satisfies $p'(\rho)>0$ and
\begin{equation}
\label{eq:pgrowth}
p'(\rho) = k \gamma \rho^{\gamma-1}  + o(\rho^{\gamma-1}) \, , \quad \hbox{as}\ \rho\to +\infty \, ,
\tag {B}
\end{equation}
with $k > 0$, $\gamma > 1$,
then $h(\rho)$ defined through $h''(\rho) = \frac{p'(\rho)}{\rho}$ verifies hypothesis \eqref{eq:growthh}
and the results of Lemma \ref{lem:generalconvEuler3d} apply in that case.
\end{remark}

\subsection{Convergence in the diffusive relaxation limit}\label{subsec:l1}
Proposition \ref{prop:relenEuler} is used in order to prove convergence from the Euler equations 
with friction in the diffusive limit towards the porous media equation. 
We carry out the analysis in two frameworks:
\begin{itemize}
\item multi-d periodic solutions;
\item 1-d solutions in the real line with (possibly distinct) constant states $\rho_{\pm}$ at $\pm\infty$.
\end{itemize}
In both cases, the main hypothesis is that  $\bar \rho$ is a smooth solution of \eqref{eq:pme} 
that sits away from vacuum.
\begin{remark}\label{rem:otherframeworks}
It is worth to observe that other possible frameworks can be analyzed with these techniques, and we restrict our ourselves to the aforementioned cases to avoid further technicalities. For instance, with small modification in the arguments below, we can consider multi-d solutions 
 $(\bar\rho, \bar m)$ such that $\bar\rho\to\rho_*>0$ as $|x|\to+\infty$;  $\bar m  = -\e\nabla p(\bar \rho)$ and such that 
 $\rho\geq 0$, $\rho - \rho_* $, $ m \in L^1([0,T]\times\R^3)$.
\end{remark}
\subsubsection{\bf Multidimensional periodic solutions}
In the periodic case, we work within the following framework, collectively
referred to as $\mathbf{(H_1)}$:
\begin{itemize}
\item[(i)]
$(\rho \, , \, m)  : (0,T)\times \T^3 \to \R^4$ is a (periodic) \emph{dissipative weak solution} of \eqref{eq:euler} 
with $\rho \geq 0$,
satisfying the weak form of  \eqref{eq:euler} and the integrated form of the entropy inequality \eqref{eq:entrEuler}:
\begin{align}
 & \iint_{[0,+\infty)\times\T^3} \left ( \frac{1}{2} \frac{|m |^2}{\rho} + h(\rho) \right ) \dot\theta(t) 
 - \frac{1}{\e^{2}}\frac{|m|^{2}}{\rho}\theta(t) \, dxdt 
 \nonumber
 \\
 &\qquad
 +   \int_{\T^3}  \left ( \frac{1}{2} \frac{|m |^2}{\rho } + h(\rho) \right ) \Big |_{t=0} 
 \theta(0)dx 
 \ge 0\, ,
 \label{dissipsol}
\end{align}
where $\theta(t)$ is a nonnegative Lipschitz test function compactly supported in $[0,T)$. 
The family $(\rho^\eps , m^\eps)$ is
assumed to satisfy the uniform bounds
\begin{align}
\sup_{t\in (0,T)} \int_{\T^3} \rho^\eps  dx &\le K_1 < \infty\, ,
\label{hypCauchy1}
\\
 \sup_{t\in (0,T)} \int_{\T^3}
 \frac{1}{2} \frac{|m^\e |^2}{\rho^\e} + h(\rho^\e ) \, dx
 &\le K_2 < \infty \, , \nonumber
\end{align}
which are natural within the given framework, and follow from corresponding uniform bounds on the initial data.

\item[(ii)]
 $\bar\rho$ is a smooth ($C^3$) periodic solution of the multidimensional porous media equation  \eqref{eq:pme} 
 that avoids vacuum, $\bar \rho \ge c > 0$; $\bar m$ is defined via $\bar m = -\e\nabla p(\bar \rho)$,
\end{itemize}

The function 
\begin{equation}
\varphi(t) = \int_{\T^3} \eta(\rho, m \left | \bar \rho, \bar m\right. ) dx 
\label{eq:norm}
\end{equation}
will be used as a measure to control the distance between two solutions.
We prove:

\begin{theorem}\label{th:finalEuler3d}
Let  $T>0$ be fixed and assume $p(\rho)$ satisfies \eqref{eq:hyppress} and \eqref{eq:pgrowth}.
Under hypothesis $\mathbf{(H_1)}$, the stability estimate
\begin{equation}
\varphi(t) \leq C \big (\varphi(0) + \e^4 \big ), \quad t\in [0,T] \, ,
\label{eq:stabEuler3d}
\end{equation}
holds,  where $C$ is a positive constant depending only on $T$, $K_1$, $\bar\rho$ and its derivatives.
Moreover, if $\varphi(0)\to 0$ as $\e \downarrow 0$, then
\begin{equation}
\sup_{t\in[0,T]}\varphi(t) \to 0, \ \hbox{as}\ \e \downarrow 0.
\label{eq:convEuler3d}
\end{equation}
\end{theorem}

\begin{proof}
We proceed to establish the integrated version of \eqref{eq:relenEuler} under 
the regularity framework $\mathbf{(H_1)}$.
To this end, we introduce in \eqref{dissipsol} the choice of test function
\begin{equation}
\theta(\tau) := 
\begin{cases}
1, &\hbox{for}\ 0\leq \tau < t, \\
\frac{t-\tau}{\kappa} + 1, &\hbox{for}\ t\leq \tau < t+\kappa , \\
0, &\hbox{for}\ \tau \geq t+\kappa .
\end{cases}
\label{testtheta}
\end{equation}
Taking the limit $\kappa \downarrow 0$ gives
\begin{align*}
& \left. \int_{\mathbb{T}^3}\left ( \frac{1}{2} \frac{|m|^2}{\rho} + h(\rho ) \right)dx  \right |^{t}_{\tau=0}
\leq  - \frac{1}{\e^{2}}\int_{0}^t\int_{\T^3}\frac{|m|^{2}}{\rho}dxd\tau \, .
\end{align*}
Next, integrating \eqref{eq:smoothentr} over $(0,t) \times \T^3$, gives
\begin{align*}
& \left. \int_{\mathbb{T}^3}\left ( \frac{1}{2} \frac{|\bar m|^2}{\bar \rho} + h(\bar \rho) \right)dx \right |^{t}_{\tau=0}
\leq  \int_{0}^t\int_{\T^3} \left ( - \frac{1}{\e^{2}} \frac{|\bar m|^{2}}{\bar \rho} + \frac{\bar m}{\bar \rho} \cdot \bar e \right ) dxd\tau \, .
\end{align*}

Finally, to justify the calculations leading to \eqref{eq:uno}, we start from the weak form of \eqref{eq:diff}:
\begin{align}
&- \iint {\phi}_t (\rho - \bar \rho) + \frac{1}{\eps} \phi_{x_i} (m_i - \bar m_i) \, dx dt
- \int \phi(x,0) (\rho - \bar \rho) \Big |_{t=0} \, dx = 0\, ,
\label{weakmass}
\\
&- \iint \psi_t \cdot (m - \bar m) + \frac{1}{\eps} \psi_{x_i} \cdot \big ( f_i (\rho, m)  - f_i (\bar \rho, \bar m) \big 
) \, dx dt
\nonumber
\\
&- \int \psi(x,0) \cdot (m - \bar m) \Big |_{t=0} \, dx 
= 
\iint \psi \cdot \left ( -\frac{1}{\eps^2} (m - \bar m) - \bar e \right ) \, dx dt\, ,
\label{weakmomentum}
\end{align}
where $\phi$, $\psi$ are Lipschitz test functions 
compactly supported in $ [0,T)\times \T^3$ 
and $\psi$ 
is vector valued.
Using the test functions
$$
\phi = \theta(\tau)\omega(x) \eta_\rho(\bar\rho,\bar m) \, , \quad
\psi = \theta(\tau)\omega(x)  \nabla_m \eta(\bar\rho,\bar m)
$$
with $\theta(\tau)$ as in \eqref{testtheta}  and $\omega (x) =1$, and upon 
taking $\kappa \downarrow 0$, this gives
\begin{align}
&\left. \int_{\T^3} \Big (\eta_{\rho}(\bar \rho,\bar m)(\rho - \bar \rho)  + 
    \nabla_{m}\eta (\bar \rho,\bar m) \cdot (m - \bar m)  
    \Big ) dx \right |^{t}_{\tau=0}
= \int_0^t \int_{\T^3} J dx dt\, ,
\label{intermed}
\end{align}
where $J$ is as in \eqref{eq:uno}. 
Combining the above inequalities leads to
\begin{equation}
 \varphi(t)+\frac{1}{\e^2}
\int_{0}^t\int_{\T^3} R(\rho, m\left | \bar \rho, \bar m \right. )dx d\tau  \leq   \varphi(0) + \int_{0}^t\int_{\T^3}(|Q|+ |E|)dx d\tau 
    \label{eq:relenEulerint}
    \end{equation}
where $Q$, $E$ and $R$ are given in \eqref{eq:defrhs}.

Using \eqref{eq:defrhs}, Remark \ref{rem:relenEulerQE} (a), Lemma \ref{lem:controlQ} and \eqref{eq:norm}, we deduce
\begin{equation*}
\int_{0}^t\int_{\T^3}|Q|dx d\tau  \leq C_1 \int_{0}^t\varphi(\tau)d\tau \, ,
\end{equation*}
where $C_1$ depends on $\| \partial_{x_ix_j}h'(\bar\rho) \|_{L^\infty}$.
The error term $E$ in \eqref{eq:defrhs} is estimated by
\begin{align*}
\int_{0}^t\int_{\T^3} |E|dx d\tau  & \leq \frac{\e^2}{2}\int_{0}^t\int_{\T^3} \left | \frac{\bar e}{\bar \rho} \right |^2\rho dx d\tau   + \frac{1}{2\e^2}\int_{0}^t\int_{\T^3}\rho \left | \frac{m}{\rho} -  
    \frac{\bar m}{\bar \rho} \right |^2 dx d\tau  
\end{align*}
and, by \eqref{eq:errortermdef} and \eqref{hypCauchy1},
\begin{equation}
\label{errorterm}
\int_{0}^t\int_{\T^3} \left | \frac{\bar e}{\bar \rho} \right |^2\rho dx d\tau \le C_1 \eps^2 \, t \, ,
\end{equation}
where $C_2$  depends on   $K_1$, $T$ and $\bar\rho$ through the following 
norms of derivatives up to third order:
\begin{equation*}
\left \|   \frac{1}{\bar \rho}
\dx \left ( \frac{\nabla_x p(\bar\rho)\otimes \nabla_x p(\bar\rho)}{\bar\rho}\right)  \right\|_{L^\infty}
\quad
 +  \left \|  \frac{1}{\bar \rho} \partial_t\nabla_x p(\bar\rho)\right\|_{{L^\infty}}.
\end{equation*}

Introducing the above estimates into \eqref{eq:relenEulerint}, we obtain 
$$
\begin{aligned}
\varphi(t) &+ \frac{1}{2\e^2} \int_0^t\int _{\T^3} R(\rho, m\left | \bar \rho, \bar m \right. ) dx d\tau  
\\
&\leq \varphi(0) + C_1\int_0^t \varphi(\tau)d\tau + 
C_2 \e^4t \,  .
\end{aligned}
$$
Gronwall's  inequality then implies
\begin{equation*}
\varphi(t) \leq C  \big (\varphi(0) +  \e^4   \big ), \quad t \in (0, T]
\end{equation*}
and  \eqref{eq:convEuler3d} follows.
\end{proof}

\subsubsection{\bf The Cauchy problem on the real line}
Next, we consider the Cauchy problem in one-space dimension 
for 
\begin{equation}
    \begin{cases}
        \displaystyle{\rho_{t} +\frac{1}{\e} m_x =0}&   \\
	& \\
       \displaystyle{ m_{t} + \frac{1}{\e}\left (\frac{m^2}{\rho} 
       + p(\rho) \right)_x
	= -\frac{1}{\e^{2}}m} \, . & 
    \end{cases}
    \label{eq:euler1d}
\end{equation}
To avoid unnecessary technicalities with the behavior as $|x| \to \infty$,  we assume the initial data 
$(\rho_0, m_0)$ take constant values outside a compact set $[-R_0. R_0]$,
$$
\begin{aligned}
( \rho_0 (x) , m_0 (x) ) &= (  \rho_- , 0) \quad \mbox{ for $x < - R_0 $}\, ,
\\
( \rho_0 (x) , m_0 (x) ) &= (  \rho_+ , 0) \quad \mbox{ for $x > R_0 $}
\end{aligned}
$$
for some $\rho_{\pm} > 0$.
By the finite speed of propagation property, any solution $(\rho, m)$ will assume the same values outside the 
cones $x < -R_0 - k t$ and for $x > R_0 + kt$, respectively, with $k$ calculated in terms of the maximum wave speed 
on the range of the data.

Let $\bar\rho > 0$ be a smooth solution of
\begin{equation}
\bar\rho_t - p(\bar\rho)_{xx} = 0
\label{eq:pme1d}
\end{equation}
with initial data $\bar \rho_0$ taking constant values
$$
\rho_0 (x)  = \rho_-  \; \; \mbox{ for $x < - R_0 $} \, , \quad
\rho_0 (x) =  \rho_+ \; \;  \mbox{ for $x > R_0 $} \, ,
$$
outside some compact set $[-R_0, R_0]$, with $\rho_{\pm} > 0$ as above. By standard theory for the
porous media equation (see \cite{Vaz07}), the solution of \eqref{eq:pme1d} satisfies $\bar \rho (x,t)  \ge c > 0$,
and satisfies $\bar\rho (x,t) \to\rho_{\pm}$ as $x\to \pm\infty$ with sufficiently fast decay   (in fact
exponential). Defining 
$\bar m = -\e p(\bar\rho)_x$ , we obtain $\bar m\to 0$ as $x\to \pm\infty$.

By modifying the entropy pair \eqref{eq:mechen}-\eqref{eq:mechflux} (using a trivial linear pair),
we define
\begin{align*}
& \tilde \eta (\rho,m) = \eta (\rho,m) - \frac{h(\rho_+) - h(\rho_-)}{\rho_+ - \rho_-}\Big (\rho - \frac{1}{2}(\rho_+ + \rho_-) \Big) 
- \frac{1}{2}(h(\rho_+) + h(\rho_-)), \\
& \tilde q (\rho,m) = q (\rho,m) - \frac{h(\rho_+) - h(\rho_-)}{\rho_+ - \rho_-}m
\end{align*}
so that $\tilde \eta (\rho_\pm,0)=0$. The resulting $(\tilde \eta - \tilde q)$ is an entropy pair, and vanishes at
the end states $(\rho_{\pm}, 0)$.

We next summarize the framework $\mathbf{(H_2)}$ for the relaxation limit:
\begin{itemize}
\item[(i)]
$(\rho \, , \, m)  : (0,T)\times \R \to \R^2$ with $\rho \ge 0$ is a \emph{dissipative weak solution} 
of \eqref{eq:euler1d}, 
that is, it satisfies the weak form of  \eqref{eq:euler1d} and the integrated form of the entropy inequality 
\begin{align*}
 & \iint_{[0,+\infty)\times\mathbb{R}}\tilde \eta(\rho, m ) \dot\theta(t)dxdt + 
\left. \int_{\mathbb{R}}\tilde \eta(\rho, m )\right |_{t=0} \theta(0)dx 
\\
  &\quad \geq  
    \frac{1}{\e^{2}}\iint_{[0,+\infty)\times\mathbb{R}}\frac{|m|^{2}}{\rho}\theta(t) dxdt  \, ,
\end{align*}
with $\theta(t)$ a non negative Lipschitz 
test function compactly supported in $[0, +\infty)$.
The family $(\rho^\eps , m^\eps)$ is assumed to satisfy the uniform bounds
\begin{align}
\sup_{t\in (0,T)} \left ( \int_{-\infty}^0  | \rho^\eps - \rho_-|   dx  + \int^{\infty}_0  | \rho^\eps - \rho_+|   dx \right ) &\le K_1 < \infty\, ,
\label{hypCauchy1d1}
\\
 \sup_{t\in (0,T)} \int_{\R}
 \tilde \eta (\rho^\eps, m^\eps)  \, dx
 &\le K_2 < \infty\, , \nonumber
\end{align}
with $K_1$, $K_2$ independent of $\eps$. Of course, this dictates analogous uniform bounds on the energy norm
of the initial data $(\rho_0^\eps, m_0^\eps)$.

\item[(ii)]
 $\bar\rho$ is a smooth ($C^3$) solution of \eqref{eq:pme1d} that satisfies $\bar \rho \ge c >0$; 
 $\bar m$ is defined via $\bar m = -\e\nabla p(\bar \rho)$.
\end{itemize}
We now denote by
\begin{equation}
\phi(t) = \int_{\R} \tilde \eta(\rho, m \left | \bar \rho, \bar m\right. ) dx.
\label{eq:norm1d}
\end{equation}
This will replace \eqref{eq:norm} as a yardstick for measuring distance between solutions in the one-dimensional
Cauchy problem. Then we have:

\begin{theorem}\label{th:finalEuler1d}
Let  $T>0$ be fixed and assume \eqref{eq:hyppress} and \eqref{eq:pgrowth} hold. If $(\rho, m)$ and $(\bar \rho, \bar m)$
are under the framework $\mathbf{(H_2)}$, the following stability estimate holds:
\begin{equation*}
\phi(t) \leq C \big (\phi(0) + \e^4 \big ), \quad t \in (0, T] \, ,
\end{equation*}
where $C$ is a constant depending only on $T$, $\rho_\pm$, $\bar\rho$ and its derivatives up to third order.
Moreover, if $\phi(0)\to 0$ as $\e \downarrow 0$, then
\begin{equation*}
\sup_{t\in[0,T]}\phi(t) \to 0, \ \hbox{as}\ \e \downarrow 0.
\end{equation*}
\end{theorem}

\begin{proof}
Proceeding along the  lines of the proof of Theorem \ref{th:finalEuler3d}, one derives 
the analog of \eqref{eq:relenEulerint} for $\phi$ in \eqref{eq:norm1d}. 
There is however a difference in the derivation as applies to the Cauchy problem: the equations
\eqref{weakmass} and \eqref{weakmomentum} hold for test functions compactly supported in
$[0,T)\times \R$. Thus we introduce the test functions
$$
\varphi = \theta(\tau)\omega(x) \eta_\rho(\bar\rho,\bar m) \, , \quad
\psi = \theta(\tau)\omega(x)  \eta_m (\bar\rho,\bar m)\, ,
$$
where $\theta(\tau)$ defined in \eqref{testtheta} and
\begin{equation*}
\label{eq:theta}
\omega(x) = 
\begin{cases}
1  &\hbox{for}\ -R < x < R, \\
1 + \frac{R - x}{\delta} &\hbox{for}\ R < x < R + \delta \\
1 + \frac{R + x}{\delta} &\hbox{for}\ - R - \delta  < x < - R \\ 
0  &\hbox{for}\ x > R+\delta \; \; \hbox{or} x < - R - \delta \, ,
\end{cases}
\end{equation*}
into \eqref{weakmass} and \eqref{weakmomentum}. Sending $R \to \infty$,
using  the asymptotic properties in $x$ of $(\rho, m)$ and $(\bar \rho, \bar m)$, and subsequently sending
$\kappa \downarrow 0$,  we obtain the analog of \eqref{intermed} and through that the analog of 
\eqref{eq:relenEulerint}.

A second difference lies in replacing \eqref{errorterm} by the estimation
$$
\begin{aligned}
\int_{0}^t\int_{\R} \left | \frac{\bar e}{\bar \rho} \right |^2\rho dx d\tau &\le 
\left \| \frac{\bar e}{\bar \rho}  \right \|^2_{L^\infty} 
\int_{0}^t\int_{\R} \left | \rho - \big ( \rho_- \charf_{x < 0}  + \rho_+ \charf_{x < 0} \big ) \right | dx d\tau 
\\
&\quad + 
\max\{ \rho_-, \rho_+\} \int_{0}^t\int_{\R} \left | \frac{\bar e}{\bar \rho} \right |^2 dx d\tau
\\
&\le C \eps^2 t\, ,
\end{aligned}
$$
where we used \eqref{eq:errortermdef} and the constant $C$  depends on $T$,  $K_1$ in \eqref{hypCauchy1d1}, 
and also on $\bar\rho$ through the $L^\infty$ norms
of space-time derivatives up to third order and the norms 
\begin{equation*}
\Big \| \frac{\dx \left ( \frac{\nabla_x p(\bar\rho)\otimes \nabla_x p(\bar\rho)}{\bar\rho}\right)}{\bar\rho} 
\Big \|_{L_t^\infty (L^2(\R)) } +  \left \|  \frac{\partial_t\nabla_x p(\bar\rho)}{\bar\rho}\right\|_{L_t^\infty (L^2(\R))}.
\end{equation*}
Again using Gronwall, we deduce
\begin{equation*}
\phi(t) \leq C  \big (\phi(0) +  \e^4   \big ), \quad t \in (0, T]
\end{equation*}
which completes the proof.
\end{proof}

\subsection{Relative entropy for entropic measure-valued solutions}
\label{subseq:relenmv}

A variant of the relative entropy identity can be derived for comparing entropic measure-valued solutions 
of \eqref{eq:euler} with smooth solutions of \eqref{eq:pme}. Such calculations are in the spirit of the recent works
\cite{BDS11, DST12}, the difference here being that two dissipative systems are compared.

Let $\nu = \big\{ \nu_{x,t} \big\}_{(x,t) \in Q_T}$ be a parametrized family of probability measures  (Young measures)
that acts on continuous functions $f(\lambda_\rho, \lambda_m)$, $(\lambda_\rho, \lambda_m) \in \R^+ \times \R^3$,
via
$$
\langle \nu_{x,t} , f \rangle = \int f(\lambda_\rho, \lambda_m) d \nu (\lambda)
$$
and   such that the integral (when defined) is measurable in $(x,t) \in Q_T$.
A measure-valued solution of \eqref{eq:euler} consists of a Young measure $\big\{ \nu_{x,t} \big\}_{(x,t) \in Q_T}$
with averages
\begin{equation}
\langle \nu_{x,t} , \lambda_\rho \rangle  = \rho \, , \quad 
\langle \nu_{x,t} , \lambda_m \rangle = m \, ,
\label{averages}
\end{equation}
that satisfies in the sense of distributions the measure-valued version of \eqref{eq:euler}
\begin{equation*}
    \begin{aligned}
    \rho_{t} +\frac{1}{\e}\dx m &= 0
 \\
      m_{t} + \frac{1}{\e}\dx \big\langle \nu_{x,t} , \frac{\lambda_m\otimes \lambda_m}{\lambda_\rho} 
       + p( \lambda_\rho) \,  I \big\rangle
	&= -\frac{1}{\e^{2}}m\, .
    \end{aligned}
\end{equation*}
The Young-measure  $\nu = \big\{ \nu_{x,t} \big\}_{(x,t) \in Q_T}$ is called an entropy measure valued solution
 if it also satisfies in the sense of distributions the averaged version of the entropy inequality
 \begin{equation}
\label{eq:mventro}
    \del_t \langle \nu_{x,t} \, , \, \eta(\lambda_\rho, \lambda_m) \rangle +
    \frac{1}{\e}\dx \langle \nu_{x,t} \, , \,  q(\lambda_\rho, \lambda_m) \rangle \leq  
    - \frac{1}{\e^{2}}   \Big\langle \nu_{x,t} \, , \,  \frac{|\lambda_m|^{2}}{\lambda_\rho} \Big\rangle \, ,
\end{equation}
for $\eta - q$ as in \eqref{eq:mechen}-\eqref{eq:mechflux}.

\begin{proposition}
Let $\nu = \big\{ \nu_{x,t} \big\}_{(x,t) \in Q_T}$ satisfying \eqref{averages} be an entropy measure-valued solution
of \eqref{eq:euler}, and let $(\bar \rho, \bar m)$ be a smooth solution of \eqref{eq:pmeref}. 
Then, we have the following averaged relative entropy inequality
\begin{equation}
\begin{aligned}
 \del_t  \langle \nu_{x,t} \, , \, \eta (\lambda_\rho , \lambda_m | \bar \rho , \bar m ) \rangle
 &+ \frac{1}{\e}\dx \langle \nu_{x,t} \, , \, q (\lambda_\rho , \lambda_m | \bar \rho , \bar m ) \rangle 
 \\
&\le - \frac{1}{\e^2}  \Big\langle \nu_{x,t} \, , \,  
     \lambda_\rho \Big |\frac{\lambda_m}{\lambda_\rho} - \frac{\bar m}{\bar \rho}  \Big |^2 \Big\rangle
 - \mathcal{Q} -\mathcal{E}\, ,
\end{aligned}
\label{eq:mvrelenEuler}
\end{equation}
where
\begin{align*}
& \mathcal{Q} =  
    \frac{1}{\e} \nabla^{2}_{(\rho,m)} \eta(\bar \rho, \bar m) 
    \begin{pmatrix} \bar \rho_{x_i}  \\  \bar m_{x_i} \end{pmatrix}  
    \cdot 
    \begin{pmatrix} 0  \\  \big\langle \nu_{x,t} ,  f_i (\lambda_\rho , \lambda_m | \bar \rho , \bar m ) \big\rangle
     \end{pmatrix}  \, ,
    \\
& \mathcal{E} =   \frac{e(\bar \rho, \bar m)}{\bar \rho} \cdot  
  \Big\langle \nu_{x,t} , \lambda_\rho \left ( \frac{\lambda_m}{\lambda_\rho} - \frac{\bar m}{\bar \rho} \right )
   \Big\rangle \, ,
\nonumber
\end{align*}    
and $e(\bar \rho, \bar m)$ is defined in \eqref{eq:errortermdef}.
\end{proposition}

\begin{proof}
We use \eqref{eq:relendef} to define the averaged relative entropy
$$
\begin{aligned}
\langle \nu_{x,t} \,  &, \, \eta (\lambda_\rho , \lambda_m | \bar \rho , \bar m ) \rangle
=
\langle \nu_{x,t} \, , \, \eta (\lambda_\rho , \lambda_m ) \rangle
- 
\eta (\bar \rho , \bar m ) 
\\
&-
\eta_\rho (\bar \rho , \bar m ) \langle \nu_{x,t} , \lambda_\rho - \bar \rho \rangle
-
\nabla_m \eta (\bar \rho , \bar m ) \cdot
\langle \nu_{x,t} , \lambda_m - \bar m \rangle\, .
\end{aligned}
$$
The inequality \eqref{eq:mvrelenEuler} is built by using \eqref{eq:mventro}, \eqref{eq:smoothentr}
and the averaged version of \eqref{eq:diff} and following verbatim the steps and calculations
in the proof of Proposition \ref{prop:relenEuler}.
\end{proof}

\section{The $p$--system with damping}\label{sec:psystem}
The $p$-system with damping in one space dimension is the system of conservation laws
 \begin{equation}
 \label{eq:main}
 \begin{aligned}
		u_{t}  - \frac{1}{\e}v_{x} &=0
\\
		v_{t} - \frac{1}{\e}\tau(u)_{x} &=  -
		\frac{1}{\e^{2}} v  \, , 
\end{aligned}
\end{equation}	
where $\tau$ satisfies $\tau'(u)>0$ to guarantee strict hyperbolicity. 
The system \eqref{eq:main}  is a model  either for elasticity with friction
or for isentropic gas dynamics in Lagrangian coordinates (denoted by $(x,t)$).
Then  $u$ stands for the strain (or the specific volume for gases), $v$ for the velocity and 
$\tau$ for the stress.

In the high friction limit $\eps \to 0$, solutions of \eqref{eq:main} 
converge towards a solution of the parabolic equation (see \cite{MMS88})
\begin{equation}
    u_{t} - \tau(u)_{xx} = 0 \, .
    \label{eq:equil}
\end{equation}
We will indicate in this section a simple proof of that convergence using the relative entropy identity.

For concreteness, we interpret \eqref{eq:main} as a model for shear motions, $u, v$ take values in $\R$.
We place the hypothesis that $\tau : \R \to \R$ satisfies $\tau'(u) > 0$ and the growth assumptions
\begin{equation}
\label{eq:laggrowth}
\tau(u)  = \pm |u|^p + o(|u|^p) \, , \quad \hbox{as}\ u\to \pm \infty \, ,
\tag {H}
\end{equation}
for some $p \ge 1$.

\subsection{Preliminaries}
The approach uses the mechanical energy
\begin{equation*}
    \mathcal{E}(u, v) = \frac{1}{2}v^{2} + W(u) \, , \quad \hbox{where} \quad W(u) = \int_0^{u}\tau(s)ds 
\end{equation*}
is the stored energy. The associated flux is 
\begin{equation*}
    \mathcal{F}(u, v) = -v \tau(u)
\end{equation*}
and they satisfy the entropy inequality
\begin{equation}\label{eq:entropy}
    \mathcal{E}(u, v)_{t} + \frac{1}{\e}\mathcal{F}(u, v)_{x}  \leq 
    -\frac{1}{\e^{2}}v^{2} \leq 0
\end{equation}
indicating the dissipation of the mechanical energy.

The minimum of the mechanical energy $\mathcal{E}(u,v)$ on the 
``equilibrium manifold'' of the relaxation process $\mathcal{M} = \{ (u,v) : v=0 \}$ 
 is achieved and is given by
\begin{equation*}
    W(u) = \mathcal{E}(u,0) = \min_{\mathcal{M}} \mathcal{E}(u,v)\, .
\end{equation*}
Moreover, solutions of  \eqref{eq:equil} satisfy the following energy estimate:
  \begin{equation*}
    \mathcal{E}(u, 0)_{t} 
= (\tau(u) \tau(u)_{x})_{x} - \big ( \tau(u)_{x} \big)^{2},
\end{equation*}
or equivalently
\begin{equation}
    \mathcal{E}(u, 0)_{t} + \mathcal{F}(u, \tau(u)_{x})_{x} = - \big ( \tau(u)_{x} \big)^{2}.
    \label{eq:entropyeq}
\end{equation}
Relation \eqref{eq:entropyeq} captures the equilibrium version of 
\eqref{eq:entropy}, as can be seen by applying the  Hilbert expansion to the relaxation system
\eqref{eq:main}.

Indeed, introducing the Hilbert expansion
\begin{align*}
     & u = u_{0} + \e u_{1} + \e^{2}u_{2}+ \ldots  \\
     & v = v_{0} + \e v_{1} + \e^{2}v_{2}+ \ldots 
\end{align*}
to \eqref{eq:main}, we see after collecting the terms of similar orders that
\begin{align*}
&O(\e^{-1}) \qquad &\px v_{0} = 0\, ,
\\
&O(\e^{0})    &\pt u_{0} + \px v_{1} = 0\, ,
\end{align*}
and
\begin{align*}
&O(\e^{-2})  &v_{0} = 0\, ,
\\
&O(\e^{-1}) \qquad &v_{1} = p(u_{0})_{x}\, ,
\\
&O(\e^{0})    & v_{2} =  (p'(u_{0})u_{1})_{x} \, .
\end{align*}
In particular, we recover the equilibrium relation $v_{0} = 0$, the 
Darcy's law $v_{1} = \tau(u_{0})_{x}$, and the diffusion
equation \eqref{eq:equil} satisfied by $u_{0}$ at equilibrium.
If the same expansion is introduced in \eqref{eq:entropy}, 
we obtain
\begin{equation*}
    \mathcal{E}(u_{0},0)_{t} -  \big ( 
    \tau(u_{0})v_{1} \big )_x =  - v_{1}^{2} 
\end{equation*}
which yields \eqref{eq:entropyeq} upon using Darcy's law $v_{1} = \tau(u_{0})_{x}$.

\subsection{Relative entropy estimate and  study of the relaxation limit}\label{subsec:relenpsystem}
To analyze the relaxation process, we consider the quadratic part of 
$\mathcal{E}(u, v)$ with respect to the ``algebraic--differential 
equilibrium'' $(\bar u, \bar v )$, where $\bar u = u_{0}$ 
and  $\bar v = \e v_{1} = \e 
\tau(\bar u)_{x}$. Namely,
\begin{align*}
    \mathcal{E}(u, v  \left | \bar u, \bar v\right. ) &= \mathcal{E} (u, v) - \mathcal{E} (\bar u, \bar v) - 
    \mathcal{E}_{u} (\bar u, \bar v) (u - \bar u) - 
    \mathcal{E}_{v} (\bar u, \bar v) (v - \bar v) \nonumber\\
    & = \frac{1}{2}(v-\bar v)^2 + W(u \left | \bar u \right.)\, .
\end{align*}
As  corresponding flux we shall consider
\begin{align*}
    \mathcal{F}(u, v  \left | \bar u, \bar v\right. ) &= \mathcal{F} (u, v) - \mathcal{F} (\bar u, \bar v) 
    + \mathcal{E}_{u} (\bar u, \bar v) ( v - \bar v)  
    + \mathcal{E}_{v}(\bar u, \bar v) (\tau(u) - \tau(\bar u)) \\
    &= - (v-\bar v)(\tau(u)-\tau(\bar u))\, .
\end{align*}
As in the previous section, to simplify the calculations, 
we  rewrite the equilibrium equation \eqref{eq:equil} as follows:
\begin{equation}
	    \begin{cases}
		\displaystyle{\bar u_{t}  - \frac{1}{\e} \bar v_{x} =0} & \\
		& \\
		\displaystyle{\bar v_{t} - \frac{1}{\e}\tau(\bar u)_{x} =  -
		\frac{1}{\e^{2}} \bar v + \bar v_{t}}\, . &
		\end{cases}
		\label{eq:equilrev}
\end{equation}	
In this way, we are able to treat the term $\bar v_{t} = \e \tau(\bar u)_{xt}$ as an error of 
order $O(\e)$.
A direct computation, along the lines  of Proposition \ref{prop:relenEuler} gives: 

\begin{proposition}\label{prop:relenpsys}
For any weak, entropy solution $(u,v)$ of \eqref{eq:main} and any smooth solution $(\bar u, \bar v)$ of \eqref{eq:equilrev} it holds:
\begin{equation}
  \mathcal{E}(u, v \left | \bar u, \bar v \right. )_t + \frac{1}{\e} \mathcal{F}(u, v\left | \bar u, \bar v \right. )_x \\
  \leq -\frac{1}{\e^2} (v-\bar v)^2  + 
 \tau(\bar u)_{xx}\tau(u \left | \bar u \right. ) - \e \tau(\bar u)_{xt}(v - \bar v)\, .
    \label{eq:relenpsys}
    \end{equation}
\end{proposition}
The terms in the right hand side of \eqref{eq:relenpsys} are analogous to the terms in \eqref{eq:relenEuler} of
 Proposition \ref{prop:relenEuler} for the Eulerian case, namely, the first term is  dissipative and is due to 
 the damping of the relaxation system relative to its diffusion limit, the second is quadratic in the flux, and the last term is a linear error term. The quadratic term is estimated with the help of the following lemma
from \cite{DST12}.

\begin{lemma}
\label{lem:lembound}
Let $\tau \in C^2 (\R)$ satisfy $\tau'(u) > 0$ and \eqref{eq:laggrowth}. If $\bar u$ takes values in a compact set $K$,
there exists constant $C$ such that
$$
\tau( u | \bar u) \le C  \, W( u | \bar u) \, , \quad \forall \;  u \in \R \, , \; \bar u \in K\, .
$$
\end{lemma}
   
\begin{proof}
Since $\bar u \in K$, $K$ a compact set, using \eqref{eq:laggrowth} there exists a constant $C$ such that
\begin{equation*}
|\tau(u \left | \bar u \right. )| \leq C ( |u|^p +1) \, , \quad u \in \R \, , \; \bar u \in K \, .
\end{equation*}
Moreover, 
\begin{equation*}
\lim_{|u|\to +\infty}\frac{W(u)}{|u|^{p+1}} = \lim_{|u|\to +\infty}\frac{\tau(u)}{(p+1)|u|^{p} \hbox{sgn} (u) } 
= \frac{1}{p+1}
\end{equation*}
and therefore, for some $c$ and $A$, we obtain
\begin{equation*}
W(u \left | \bar u \right. ) = W(u) - W(\bar u) - W'(\bar u)(u-\bar u) \geq c |u|^{p+1} - A\, ,
\end{equation*}
for $u \in \R$, $\bar u\in K$. 
We select  $U_0$ so that for $|u| > U_0$ we get $ c |u|^{p+1} - A \geq C ( |u|^p +1)$. Then
\begin{equation*}
|\tau(u \left | \bar u \right. )| \leq C W(u \left | \bar u \right. ) \quad \forall \; |u| > U_0 \, , \; \bar u \in K \, .
\end{equation*}

On the complementary interval $u \in [-U_0 , U_0]$, we have
\begin{equation*}
\tau'(u)\geq c >0
\end{equation*}
and thus
\begin{equation*}
\tau(u | \bar u) \le  C_1 (u-\bar u)^2 \leq  C_2 W(u | \bar u)  \qquad \forall \; |u| \le U_0 \, , \; \bar u \in K \, .
\end{equation*}
\end{proof}

Using Proposition \ref{prop:relenpsys}, we obtain the main stability and convergence result in terms of  the quantity
\begin{equation*}
\Phi(t) = \int_{\R}  \mathcal{E}(u, v \left | \bar u, \bar v\right. )dx.
\end{equation*}

\begin{theorem}\label{th:finalpsys}
Assume $\tau$ satisfies $\tau' > 0$ and \eqref{eq:laggrowth}.
Let $\bar u$ be a smooth solution of \eqref{eq:equil}, defined on $Q_T = \R \times [0,T)$ with $T > 0$ fixed,
$\bar v = \eps \tau (\bar u)_x$, be 
 such that $\bar u$ is bounded and $\| \tau(\bar u)_{x t}\|_{L^2(Q_T)} \le K < \infty$.
Let $(u,v)$ be a weak, entropy solution  of \eqref{eq:main} such that 
$\Phi(0) < +\infty$
and
\begin{equation*}
  \mathcal{F}(u, v \left | \bar u, \bar v\right. )\to 0,\ \hbox{as}\ x\ \to \pm\infty.
\end{equation*}
Then the following stability estimate holds:
\begin{equation}
\Phi(t) \leq C (\Phi(0) + \e^4), \quad t \in [0, T) \, ,
\label{eq:stabpsys}
\end{equation}
where $C$ is a constant depending on $T$, the properties of $\tau (u)$, and the function
$\bar u$ and its derivatives.
Moreover, if $\Phi(0)\to 0$ as $\e \downarrow 0$, then
\begin{equation*}
\sup_{t\in[0,T]}\Phi(t) \to 0, \ \hbox{as}\ \e \downarrow 0.
\end{equation*}
\end{theorem}
\begin{proof}
The proof proceeds along the lines of  Theorem \ref{th:finalEuler3d} and Theorem \ref{th:finalEuler1d}; here we shall just sketch it. 

We integrate \eqref{eq:relenpsys} over $\R \times [0,t]$, $t < T$.  The right hand side of \eqref{eq:relenpsys}
is estimated using Lemma \ref{lem:lembound}, 
\begin{equation*}
| \tau(\bar u)_{xx}\tau(u \left | \bar u \right. )| \leq  C_1 W(u|\bar u)   \leq C \mathcal{E}(u, v \left | \bar u, \bar v \right. ) \, , 
\end{equation*}
and Young's inequality
\begin{equation*}
|\e \tau(\bar u)_{xt}(v - \bar v)| \leq \frac{1}{2\e^2}(v-\bar v)^2 + \frac{\e^4}{2}  \big | \tau(\bar u)_{xt} \big |^2 \, .
\end{equation*}
Then, \eqref{eq:stabpsys} is a direct consequence of the Gronwall Lemma.
\end{proof}

\section{Viscoelasticity with memory}\label{sec:viscoelasticity}
We conclude with an example where the diffusive scaling limit is  a  hyperbolic -- parabolic system. 
Consider the following $3\times 3$,
    one dimensional, quasilinear system of viscoelasticity with memory effects:
\begin{equation}\label{eq:mainVE}
\begin{aligned}
u_{t}  - v_{x} &=0
\\
v_{t}  - \sigma(u)_{x} - \frac{1}{\e}  z_{x} &=0
\\
z_{t} - \frac{\mu}{\e}v_{x}  &=  - \frac{1}{\e^{2}} z\, ,
\end{aligned}
\end{equation}	
where $\mu > 0$ and the elastic stress function $\sigma$ satisfies the usual 
condition $\sigma'(u)>0$ which guarantees 
 hyperbolicity. 
The above system describes a one dimensional viscoelastic material for 
which the stress $S = \sigma(u) + \frac{1}{\e} z$ is the sum of an elastic part and a 
viscoelastic part of the memory type (see \eqref{eq:mainVE}$_{3}$).
The system is scaled appropriately so that it relaxes as $\eps \to 0$ to the 
equations of viscoelasticity of the rate type, 
\begin{equation}\label{eq:equilVE}
\begin{aligned}
u_{t}  - v_{x} &=0
\\
v_{t}  - \sigma(u)_{x} &=\mu v_{xx} \, .
\end{aligned}
\end{equation}	
In the latter system, the total stress
 $T = \sigma(u) + \mu v_{x}$ consists of an elastic part and a Newtonian viscous stress.
 We refer to \cite{DiFL04,DL09} for studies of a corresponding semilinear relaxation framework,
using energy bounds. Here, we focus at the quasilinear level, and pursue a relative entropy
analysis to explore the relation between the two systems.

The mechanical energy for 
\eqref{eq:mainVE} is 
\begin{equation*}
    \mathbb{E}(u,v,z) = \int_0^{u}\sigma(s)ds + \frac{1}{2}v^{2} + 
    \frac{1}{2\mu} z^{2} = \Sigma(u) + \frac{1}{2}v^{2} + 
    \frac{1}{2\mu} z^{2},
\end{equation*}
with  energy flux 
\begin{equation*}
    \mathbb{F}_{\e}(u,v,z) = - (\e \sigma(u)v +vz)\, .
\end{equation*}
Weak solutions of \eqref{eq:mainVE} are required to satisfy
the entropy inequality
\begin{equation}\label{eq:entropyVE}
     \mathbb{E}(u,v,z)_{t} + \frac{1}{\e} \mathbb{F}_{\e}(u,v,z)_{x}  \leq 
    -\frac{1}{\mu \e^{2}}z^{2} \, ,
\end{equation}
manifesting the dissipation of the mechanical energy. (Smooth solutions of \eqref{eq:mainVE}
satisfy \eqref{eq:entropyVE} as an equality.)

The equation capturing the dissipation of mechanical energy for smooth processes of 
viscoelasticity of the rate type \eqref{eq:equilVE} reads
\begin{equation}
     \mathbb{E}(u,v,0)_{t} +  \mathbb{F}_{1}(u, v, \sigma(u)_{x})_{x} = - \mu (v_{x})^{2},
    \label{eq:entropyeqVE}
\end{equation}
where
\begin{equation*}
     \mathbb{E}(u,v,0) =  \Sigma(u)+ \frac{1}{2}v^{2} 
\end{equation*}
is the equilibrium energy for $\mathbb{E}(u,v,z)$ and 
\begin{equation*}
    \mathbb{F}_{1}(u, v, \sigma(u)_{x}) = - ( \sigma(u)v +\mu vv_{x}).
\end{equation*}
Note that \eqref{eq:entropyeqVE} is the leading order (with respect to the relaxation parameter)
asymptotic development of the energy dissipation inequality \eqref{eq:entropyVE}.
This may be seen, as in the previous sections,  by expanding \eqref{eq:entropyVE}  in terms of the 
Hilbert expansion; we omit the details here.

\subsection{Relative entropy estimate and study of the relaxation limit}\label{subsec:relenexVE}
Following the general procedure, outlined in Section \ref{subsec:relenexEuler}, 
we  recast the equilibrium system \eqref{eq:equilVE} and the 
corresponding stress--strain response in the variables
$(\bar u, \bar v, \bar z)$ with $\bar z = \eps \mu \bar v_x$ as 
follows:
\begin{equation}
    \begin{cases}
	\displaystyle{\bar u_{t} - \bar v_{x} =0}&   \\
	& \\
       \displaystyle{ \bar v_{t} - 
       \sigma(\bar u)_{x} - 
	\frac{1}{\e}\bar z_{x} =0} & \\
	& \\
	\displaystyle{\bar z_{t} - \frac{\mu}{\e}\bar v_{x} =  -
		\frac{1}{\e^{2}} \bar z + \bar z_{t}}\, ,
    \end{cases}
    \label{eq:VEbar}
\end{equation}
where we shall treat the term $\bar z_{t}$ as an $O(\e)$ error:
\begin{equation*}
    \bar z_{t} = \e \mu \bar v_{xt} = \e\mu \big (\sigma(\bar u)_{x} 
    + \mu \bar v_{xx} \big)_{x}\, .
\end{equation*}
We define the relative entropy and relative entropy flux, respectively,
\begin{align*}
     &\mathbb{E}(u,v,z\left | \bar u, \bar v, \bar z \right. ) =  \mathbb{E}( u,v,z) -  \mathbb{E}( \bar u, \bar v, \bar z) 
\\
&\quad - \mathbb{E}_{u} (\bar u, \bar v, \bar z) (u - \bar u) - 
     \mathbb{E}_{v} (\bar u, \bar v, \bar z)(v - \bar v) - 
     \mathbb{E}_{z}(\bar u, \bar v, \bar z)(z-\bar z)\, ,
\\
     &\mathbb{F}_{\e}(u,v,z\left | \bar u, \bar v, \bar z \right. ) 
     = \mathbb{F}_{\e} (u,v,z) - \mathbb{F}_{\e}(\bar u, \bar v, \bar z) - 
    \mathbb{E}_{u} (\bar u, \bar v, \bar z) \big (-\e (v - \bar v)\big) 
\\
    &\quad -\mathbb{E}_{v} (\bar u, \bar v, \bar z)\big(-\e (\sigma(u) - 
	\sigma(\bar u)) - (z -\bar z)\big) 
	 - 
	\mathbb{E}_{z}(\bar u, \bar v, \bar z)\big(v-\bar v\big)\, ,
\end{align*}
and derive the relative entropy identity:

\begin{proposition}\label{prop:relenVE}
Let $(u,v,z)$  be a weak entropy solution of \eqref{eq:mainVE} and let
$(\bar u, \bar v, \bar z)$ be a smooth solution of \eqref{eq:VEbar}. Then
\begin{equation}
\begin{aligned}
  \del_t \mathbb{E}(u, v, z \left | \bar u, \bar v , \bar z\right. ) &+ \frac{1}{\e} \del_x 
  \mathbb{F}_\e(u, v, z \left | \bar u, \bar v , \bar z\right. )
  \\
  &\leq -\frac{1}{\mu\e^2} (z-\bar z)^2  + 
\bar v_x\sigma(u \left | \bar u \right. ) - \e \bar v_{xt}(z - \bar z)\, .
\end{aligned}
\label{eq:relenVE}
\end{equation}
\end{proposition}

\begin{proof}
We  use \eqref{eq:entropyVE} and rewrite 
\eqref{eq:entropyeqVE} as follows:
\begin{equation*}
    \mathbb{E}(\bar u, \bar v, \bar z)_{t} +\frac{1}{\e}\mathbb{F}_\e(\bar u, \bar 
    v, \bar z)_{x}
    = 
    -\frac{1}{\mu\e^{2}} \bar z^{2} + 
     \e \bar z\bar v_{xt}\, .
\end{equation*}
Moreover, a direct computation shows
\begin{align*}
    & \pt \Big (\sigma(\bar u)(u - \bar u)  + 
    \bar v (v - \bar v) + \frac{1}{\mu}\bar z (z-\bar z)  
    \Big )  \\
    & \ \ + \frac{1}{\e}  \px \Big (-\e\sigma(\bar u)(v-\bar v) - 
    \e\bar v(\sigma(u)-\sigma(\bar u)) - \bar v(z-\bar z) - \bar 
    z(v-\bar v) \Big) \\
	&\ = - \frac{1}{\mu\e^{2}}\bar z (z-\bar z) - 
	\e \bar z  \bar v_{xt} + \sigma'(\bar u)\bar 
	u_{t}(u-\bar u) + \bar v_{t}(v-\bar v) + \frac{1}{\mu}\bar 
	z_{t} (z -\bar z) \\
	&\ \ - \sigma(\bar u)_{x}(v-\bar v) - \bar 
	v_{x}(\sigma(u) - \sigma (\bar u)) - \frac{1}{\e}\bar 
	v_{x}(z-\bar z) - \frac{1}{\e}\bar z_{x}(v-\bar v) \\
	&\ = - \frac{2}{\mu\e^{2}}\bar z (z-\bar z) 
	   + \e \bar v_{xt} (z-2\bar z)
	  - \bar v_{x} \big (\sigma(u) - \sigma (\bar u) - 
	\sigma'(\bar u)(u-\bar u)\big)\, .
\end{align*}
Finally, putting all relations together we obtain \eqref{eq:relenVE}.
\end{proof}

Proposition \ref{prop:relenVE} suggests to measure the distance between systems 
\eqref{eq:mainVE} and \eqref{eq:equilVE} via the quantity
\begin{equation*}
\Psi(t) = \int_{\R}  \mathbb{E}(u,v,z\left | \bar u, \bar v, \bar z \right. )dx\, .
\end{equation*}
For the stress function $\sigma \in C^2(\R)$ we assume
the hypotheses
\begin{align}
\sigma'(u)  &> 0\, ,
\tag{a$_1$}
\label{hypa1}
\\
\sigma(u)  &= \pm |u|^p + o(|u|^p) \, , \quad \mbox{as  $u\to \pm \infty$, for some $p \ge 1$} \, ,
\tag {a$_2$}
\label{hypa2}
\end{align}
and show the following stability and convergence results:

\begin{theorem}\label{th:finalVE}
Let  $T>0$ be fixed and let $(\bar u, \bar v, \bar z)$ be any smooth solution of \eqref{eq:VEbar} with in particular
 $\bar u$ bounded in $L^\infty$.
Let $(u,v,z)$ be a weak, entropy solution  of \eqref{eq:mainVE} such that
$\Psi(0)< +\infty$
and
\begin{equation*}
  \mathbb{F}_\e(u,v,z\left | \bar u, \bar v, \bar z \right. )\to 0,\ \hbox{as}\ x\ \to \pm\infty\, .
\end{equation*}
If $\sigma$ verifies \eqref{hypa1}, \eqref{hypa2}, then 
\begin{equation*}
\Psi(t) \leq C (\Psi(0) + \e^4), \quad t\in [0,T)\, , 
\end{equation*}
for a given positive constant $C$ depending only on $T$, the properties of $\sigma$
and the limit functions $\bar u$, $\bar v$ and their derivatives. 
Moreover, if $\Psi(0)\to 0$ as $\e \downarrow 0$, then
\begin{equation*}
\sup_{t\in[0,T]}\Psi(t) \to 0, \ \hbox{as}\ \e \downarrow 0\, .
\end{equation*}
\end{theorem}

The proof employs the relative entropy inequality \eqref{eq:relenVE} and  proceeds following 
Theorems \ref{th:finalpsys} and {\ref{th:finalEuler3d}; the details are omitted here.

\bigskip
\noindent
{\bf Acknowledgements.} 
Research partially supported by the EU FP7
REGPOT project 
{\it ``Archimedes Center for
Modeling, Analysis and Computation''}.
AET partially supported by the ARISTEIA program of the Greek Secretariat
for Research.


\begin{thebibliography}{99}

\bibitem{BV05}
{\sc Berthelin, F., Vasseur A.},
 {\em From kinetic equations to multidimensional isentropic gas dynamics
  before shocks},
 {SIAM J. Math. Anal.},  36 (2005), pp.~1807-1835.

\bibitem{BDS11}
{\sc Brenier, Y.,  De Lellis, C., and L. Sz\`ekelyhidi Jr.},
{\em Weak-strong uniqueness for measure-valued solutions},
Comm. Math. Physics,
 305  (2011), pp.~351-361.

\bibitem{CFL02}
{\sc Chen, G-Q., Frid, H., and Li Y.},
{\em Uniqueness and stability of Riemann solutions with large
     oscillation in gas dynamics},
     {Commun. Math. Phys.},  {228} (2002), pp.~201-217.
   

\bibitem{CG07}
{\sc Coulombel, J.-F.,  Goudon, T.},
{\em The strong relaxation limit of the multidimensional isothermal Euler equations}, 
{Trans. Amer. Math. Soc.} {359} (2007), no.\ 2, pp.~637-648.

\bibitem{Dafermos79a}
{\sc Dafermos, C.M.},
{\em Stability of motions of thermoelastic fluids},
{J. Thermal Stresses} {2} (1979), pp.~127-134.

\bibitem{Dafermos79}
{\sc Dafermos, C.M.},
{\em The second law of thermodynamics and stability}, 
{Arch. Rational Mech. Anal.} {70} (1979), pp.~167-179.

\bibitem{Daf10}
{\sc Dafermos, C.M.},
{\em Hyperbolic Conservation Laws in Continuum Physics},
Third Edition.
Grundlehren der Mathematischen Wissenschaften, 325.
Springer Verlag, Berlin, 2010.

\bibitem{DST12}
{\sc Demoulini, S., Stuart, D.M.A., Tzavaras, A.E},
{\em Weak-strong uniqueness of dissipative measure-valued solutions for polyconvex
elastodynamics},
{Arch. Rational Mech. Anal.} 205 (2012), pp.~927-961.


\bibitem{DiFL04}
{\sc Di Francesco, M., Lattanzio, C.},
{\em Diffusive relaxation $3\times3$ model for a system of viscoelasticity},
{Asymptot. Anal.} {40} (2004), pp.~235-253.

\bibitem{Diperna79}
{\sc DiPerna, R.J.},
{\em Uniqueness of solutions to hyperbolic conservation laws},
{Indiana Univ. Math. J.} {28} (1979), pp.~137-188.

\bibitem{DL09}
{\sc Donatelli, D., Lattanzio, C.},
{\em On the diffusive stress relaxation for multidimensional viscoelasticity},
{Commun. Pure Appl. Anal.} {8} (2009) no.\ 2, pp.~645-654.

\bibitem{DM04}
{\sc Donatelli, D., Marcati, P.},
{\em Convergence of singular limits for multi-D semilinear hyperbolic systems to parabolic systems},
{Trans. Amer. Math. Soc.} {356} (2004), no.\ 5, pp.~2093-2121.

\bibitem{HL92}
{\sc Hsiao, L., Liu, T.-P.},
{\em Convergence to nonlinear diffusion waves for solutions of a system of hyperbolic conservation laws with damping}, 
{Comm. Math. Phys.} {143} (1992), pp.~599-605.

\bibitem{HMP05}
{\sc Huang, F., Marcati, P., Pan, R.},
{\em Convergence to the Barenblatt solution for the compressible Euler equations with damping and vacuum},
{Arch. Ration. Mech. Anal.} {176} (2005), pp.~1-24.

\bibitem{HPW11}
{\sc Huang, F., Pan, R., Wang, Z.},
{\em $L^1$ convergence to the Barenblatt solution for compressible Euler equations with damping},
{Arch. Ration. Mech. Anal.} {200} (2011), pp.~665-689.

\bibitem{Iesan94}
{\sc Iesan, D.}
{\em On the stability of motions of thermoelastic fluids},
{J. Thermal Stresses} {17} (1994), pp.~409-418.

\bibitem{LT06}
{\sc Lattanzio, C., Tzavaras, A.E.},
{\em Structural properties of stress relaxation and convergence from viscoelasticity to polyconvex elastodynamics},
{Arch. Rational Mech. Anal.}, {180} (2006), pp.~449-492. 

\bibitem{LY01}
{\sc Lattanzio, C., Yong, W.-A.},
{\em Hyperbolic-parabolic singular limits for first-order nonlinear systems}, 
{Comm. Partial Differential Equations}  {26} (2001), no.\ 5-6, pp.~939-964.

\bibitem{LV11}
{\sc Leger, N., Vasseur, A.},
{\em Relative entropy and the stability of shocks and contact discontinuities for systems of conservation
laws with non-$BV$ perturbations},
{Arch. Rational Mech. Anal.}, {201} (2011), pp.~271 - 302. 

\bibitem{LC11}
{\sc Lin, C., Coulombel, J.-F.},
{\em The strong relaxation limit of the multidimensional Euler equations},
{NoDEA Nonlinear Differential Equations Appl.}, (2012).

\bibitem{Liu97} 
{\sc Liu, T.-P.},
{\em Compressible flow with damping and vacuum},
{Japan J. Indust. Appl. Math}. {13} (1996), pp.~25-32.

\bibitem{MM90}
{\sc Marcati, P., Milani, A.J.},
{\em The one-dimensional Darcy's law as the limit of a compressible Euler flow},
{J. Differential Equations} {84} (1990), no.\ 1, pp.~129-147. 

\bibitem{MMS88}
{\sc Marcati, P., Milani, A.J., Secchi, P.},
{\em Singular convergence of weak solutions for a quasilinear nonhomogeneous hyperbolic system}, 
{Manuscripta Math.} {60} (1988), no.\ 1, pp.~49-69. 

\bibitem{MR00}
 {\sc Marcati, P., Rubino, B.},
 {\em Hyperbolic to parabolic relaxation theory for quasilinear first order systems},
 {J. Differential Equations} {162} (2000), no.\ 2, pp.~359-399. 

\bibitem{Nis96}
{\sc Nishihara, K.},
{\em Convergence rates to nonlinear diffusion waves for solutions of system of hyperbolic conservation laws with damping}, 
{J. Differential Equations}, {131} (1996), no.\ 2, pp.~171-188.


\bibitem{Ott01}
{\sc Otto, F.},
{\em The geometry of dissipative evolution equations: the porous medium equation},
{Comm. Partial Differential Equations} {26} (2001), no.\ 1-2, pp.~101-174.


\bibitem{Tza05}
{\sc Tzavaras, A.E.},
{\em Relative entropy in hyperbolic relaxation},
{Commun. Math. Sci.} {3} (2005), no.\ 2, pp.~119-132.

\bibitem{Vaz07}
{\sc V\'azquez, J.L.},
{\em The porous medium equation. Mathematical theory.}
Oxford Mathematical Monographs. The Clarendon Press, Oxford University Press, Oxford, 2007.
\end{thebibliography}
\end{document}